\documentclass[11pt]{article}
\usepackage{amsfonts, latexsym, amssymb,amsmath,amstext}
\usepackage[dvips]{graphics,color}
\usepackage[mathscr]{eucal}
\newenvironment{proof}{\par\noindent{\bf Proof.}\ }{\hfill$\Box$\par\medskip}
\newtheorem{theorem}{Theorem}[section]
\newtheorem{lemma}[theorem]{Lemma}

\newtheorem{remark}[theorem]{Remark}

\newcommand{\equal}{&\!\!\!=\!\!\!&}

\newcommand{\LD}[1]{\left[#1\right]}
\newcommand{\CC}{\mathbb C}
\newcommand{\PP}{\mathbb P}

\newcommand{\ZZ}{\mathbb Z}
\newcommand{\fK}{{\Bbb K}}
\newcommand{\hpg}[5]{{}_{#1}\mbox{\rm F}_{\!#2}\!
  \left(\left.{#3 \atop #4}\right| #5 \right) }

\makeatletter
\@addtoreset{equation}{section}

\makeatother

\title{Schlesinger transformations for algebraic Painlev\'e VI solutions}
\author{Raimundas Vidunas\thanks{Supported by the 21 Century COE Programme
"Development of Dynamic Mathematics with High Functionality" of the Ministry
of Education, Culture, Sports, Science and Technology of Japan. 
E-mail: {\sf rvidunas@gmail.com}}\; 
and\; Alexander~V.~Kitaev\thanks{Supported by
JSPS grant-in-aide no.~$14204012$. 
E-mail: {\sf kitaev@pdmi.ras.ru}}\\
Department of Mathematics, Kyushu University, 812-8581 Fukuoka, Japan\footnotemark[1]\\
Steklov Mathematical Institute, Fontanka 27, St. Petersburg, 191023, Russia\footnotemark[2]\\
and\\
School of Mathematics and Statistics, University of Sydney,\\
Sydney, NSW 2006, Australia\footnotemark[1]\;\,\footnotemark[2]}

\date{} 

\begin{document}

\maketitle

\begin{abstract}
Various Schlesinger transformations can be combined with a direct pull-back of a hypergeometric 
$2\times 2$ system to obtain $RS^2_4$-pullback transformations to isomonodromic $2\times 2$ Fuchsian systems with 4 singularities. The corresponding Painlev\'e VI solutions are algebraic functions,
possibly in different orbits under Okamoto transformations. 
This paper demonstrates a direct computation of Schlesinger transformations
acting on several apparent singular points, and presents an algebraic procedure (via syzygies)
of computing algebraic Painlev\'e VI solutions without deriving full RS-pullback transformations.
\vspace{24pt} 
\vspace{24pt}\\
{\bf Short title}: {$RS$-pullback transformations}
\end{abstract}

\newpage

\section{Introduction}

General pullback transformations of differential 
systems $d\Psi(z)/dz=M(z)\Psi(z)$ have the following general form:
\begin{equation}  \label{eq:rstrans}
z\mapsto R(x),\qquad \Psi(z)\mapsto S(x)\,\Psi(R(x)),
\end{equation}
where $R(x)$ is a rational function of $x$, and $S(x)$ is 
a linear transformation of function vectors. The transformed equation is
\begin{equation} \label{eq:pbacked}
\frac{d\Psi(x)}{dx}=\left(\frac{dR(x)}{dx}\,S^{-1}(x)M(R(x))S(x)-S^{-1}(x)\frac{dS(x)}{dx}\right) \Psi(x).
\end{equation}
The 
transformation by $S(x)$ is analogous here to
{\em projective equivalence} transformations $y(x)\to\theta(x)y(x)$ of ordinary differential equations.
If $S(x)$ is the identity transformation, we have a {\em direct pullback} of a differential equation.
For transformations to parametric (say, isomonodromic) equations, 
$R(x)$ and $S(x)$ may depend algebraically on parameter(s). 

If the equation $d\Psi(z)/dz=M(z)\Psi(z)$ is a Fuchsian isomonodromic system, one often considers
a {\em Schlesinger transformation} for $S(x)$, whereby the local monodromy difference at any $x$-point
is shifted by an integer. For example, $S(x)$ may be designed to remove apparent singularities of
the direct pullback with respect to $R(x)$. In this context, pullback transformations (\ref{eq:rstrans})
are called {\em $RS$-transformations} in \cite{K2}, \cite{HGBAA}, stressing the composition of 
a rational change of the independent variable $z\mapsto R(x)$ and 
the Schlesinger transformation $S(x)$. To merge terminology, 
we refer to these pullback transformations as {\em $RS$-pullbacks},
or {\em $RS$-pullback transformations}.

The subject of this article is construction of Schlesinger $S$-transformations 
for the $RS$-pullback transformations of $2\times2$ matrix hypergeometric equations
to isomonodromic $2\times2$ Fuchsian systems with 4 singular points. 
Corresponding solutions of the sixth Painlev\'e equation are algebraic functions, since they are determined algebraically by matrix entries of pullbacked equations (\ref{eq:pbacked}) while
those entries are algebraic functions in $x$ and the isomonodromy parameter. The second author conjectured in \cite{HGBAA} that all algebraic solutions of the sixth Painlev\'e equation can be obtained by $RS$-pullback transformations of matrix hypergeometric equations, up to Okamoto transformations \cite{O}. 

Computation of $S$-parts of suitable $RS$-transformations 
to $2\times2$ Fuchsian systems with 4 singular points does not look hard in principle.
However, this problem is not as straightforward as finding suitable projective equivalence
transformations for scalar differential equations. This article demonstrates computation of 
$RS$-transformations by several detailed examples. We use two coverings $z=R(x)$ 
computed in \cite{HGBAA}; our full $RS$-coverings are already implied there. 

In this paper, we construct a desired Schlesinger transformation at once, instead of composing several simple Schlesinger transformations (each shifting just two local monodromy differences) as was done
in \cite{AK1}, \cite{AK2}, \cite{JM}. In particular, we avoid factorization of high degree polynomials when shifting local monodromy differences at all conjugate roots by the same integer. In the context of isomonodromy problems, this approach is adopted in \cite{FN} as well.

An important observation is that the same rational covering $R(x)$ can be used in
several $RS$-pullback transformations. Application of different $RS$-transformations to respectively different matrix differential equations gives different algebraic Painlev\'e VI solutions.
For example, \cite{KV3} demonstrates usage of
the same degree 10 covering to pullback three different hypergeometric equations
(with the local monodromy differences $1/2$, $1/3$, $k/7$ with $k=1,2$ or $3$) 
and obtain three algebraic Painlev\'e VI solutions unrelated by fractional-linear or Okamoto transformations. 

In our other concrete examples, we start with matrix hypergeometric equations with the icosahedral monodromy group.  The pullbacked Fuchsian equations have the icosahedral monodromy group as
well. Corresponding Painlev\'e VI solutions are called {\em icosahedral} \cite{Bo2}. 
There are 52 types of icosahedral Painlev\'e VI solutions in total \cite{Bo2},
up to branching representation of the icosahedral monodromy group 
(or equivalently, up to Okamoto transformations). 
We recompute icosahedral solutions of Boalch types 26, 27, 31, 32. 

Second order ordinary Fuchsian equations (or $2\times 2$ first order matrix Fuchsian equations)
with a finite monodromy group are always pullbacks of a standard hypergeometric equation with the same monodromy group, as asserted by celebrated Klein's theorem \cite{Klein84}.
In particular, existence of pull-back transformations for the four icosahedral examples follows from
Klein's theorem.  R.~Fuchs \cite{F2} soon considered extension of Klein's theorem to algebraic solutions of Painlev\'e equations. Recently, Ohyama  and Okumura 
\cite{Oh} showed that algebraic solutions of Painlev\'e equations from the first to the fifth do arise from pull-back transformations of confluent hypergeometric equations,  affirming the formulation of R.~Fuchs.
The pullback method for computing algebraic Painlev\'e VI solutions
was previously suggested in \cite{HGBAA}, \cite{AK2},  \cite{K2}, \cite{Do}.
The alternative representation-theoretic approach is due to Dubrovin-Mazzocco \cite{DM}.
Recently, it was used  \cite{LosTyk} to complete classification of algebraic Painlev\'e VI solutions.

The article is organized as follows. Section \ref{sec:pbmaps} presents two almost Belyi coverings we employ. The coverings have degree 8 and 12; they were previously used in \cite{HGBAA}. 
Section \ref{sec:fullrs} demonstrates two examples of full $RS$-pullback transformations, both with respect to the degree 8 covering. 
In Section \ref{sec:generalrs} we formulate basic algebraic facts useful in computations of
$RS$-pullback transformations. 
Section \ref{sec:fullrs2} gives a direct formula for some algebraic Painlev\'e VI solutions, with minimum information from full $RS$-transformations. In Section \ref{sec:algsols}, the remaining examples of Painlev\'e VI solutions are computed. The Appendix presents the Jimbo-Miwa correspondence
between solutions of the sixth Painlev\'e equation and isomonodromic Fuchsian $2\times 2$
systems, and the matrix hypergeometric equation. In particular, the notations
$P_{VI}(\nu_0,\nu_1,\nu_t,\nu_\infty;t)$ and $E(\nu_0,\nu_1,\nu_t,\nu_\infty;y(t);z)$
for the Painlev\'e VI equation and corresponding isomonodromic Fuchsian systems 
are introduced in the Appendix.

The authors prepared {\sf Maple 9.5} worksheets supplementing this article and \cite{KV2},  \cite{KV3},
with the formulas in {\sf Maple} input format, and demonstration of key
computations. To access the worksheet, readers may contact the authors,
or search a current website of the first author on the internet.

\section{Almost Belyi coverings}
\label{sec:pbmaps}

First we introduce notation for ramification patterns, and later for $RS$-transformations.
A ramification pattern for an almost Belyi covering of degree $n$ is denoted by
$R_4(P_1|\,P_2|\,P_3)$, where $P_1,P_2,P_3$ are three partitions of $n$
specifying the ramification orders above three points. The ramification pattern
above the fourth ramification locus is assumed to be $2+1+1+\ldots+1$.
By {\em the extra ramification point} we refer to the simple ramification point in the fourth fiber.
The Hurwitz space for such a ramification pattern is generally one-dimensional 
\cite[Proposition 3.1]{zvonkin}. 

We use only genus 0 almost Belyi coverings, and write them as $\PP_x^1\to\PP_z^1$, 
meaning that the projective line with the projective coordinate $x$ is mapped to the projective line
with the coordinate $z$. Then the total number
of parts in $P_1$, $P_2$, $P_3$ must be equal to $n+3$, according to \cite[Proposition 2.1]{HGBAA};
this is a consequence of Riemann-Hurwitz formula.


We use almost Belyi coverings with the following ramification patterns:
\begin{eqnarray} \label{eq:ramif8}
& R_4\big(5+1+1+1\;|\;2+2+2+2\;|\;3+3+2\big),\\ 
\label{eq:ramif12a}
& R_4\big(3+3+3+3\;|\;2+2+2+2+2+2\;|\;5+4+1+1+1\big).
\end{eqnarray}
The degree of the coverings is 8 and 12, respectively. 
For each covering type, the three specified fibers with ramified points
can be brought to any three distinct locations 
by a fractional-linear transformation of $\PP^1_z$. We assign the first partition to $z=0$,
and the next two partitions --- to $z=1$ and $z=\infty$ respectively.
Similarly, by a fractional-linear transformation of $\PP^1_x$ we may choose any three 
$x$-points\footnote{Strictly speaking, the $x$-points in our settings are curves, or branches, parametrized by an isomonodromy parameter $t$ or other parameter, since the Hurwitz spaces
for almost Belyi maps are one-dimensional. For simplicity, we ignore the dimensions introduced
by  such parameters, and consider a one-dimensional Hurwitz space as a generic point.}
as $x=0$, $x=1$, $x=\infty$. 

For direct applications to the Painlev\'e VI equation, it is required to normalize
the point above $z=\infty$ with the deviating ramification order 2, 4 (respectively)
and the three nonramified points above $\{0,1,\infty\}\subset\PP^1_z$ as
$x=0$, $x=1$, $x=\infty$, $x=t$. We refer to explicit almost Belyi coverings normalized
this way as {\em properly normalized}. 

Properly normalized coverings with ramification patterns (\ref{eq:ramif8})--(\ref{eq:ramif12a}) were first
computed in \cite{HGBAA}. In computation of $RS$-transformations, compact expressions 
for non-normalized coverings are more convenient to use. The coverings can be computed 
on modern computers either using the most straightforward method, or an improved method \cite{KV2}
that uses differentiation. 
Here we present just explicit expressions for the almost Belyi coverings.

The degree 8 covering is:
\begin{equation} \label{eq:phi8}
\varphi_8(x)=
\frac{(s+1)^2\,x^5\left(9(s+1)^2x^3-24s(s+3)x^2+8s(11s-1)x+48s^2\right)}{64\,s\,\left(x^2-2sx-s\right)^3}.
\end{equation}
The Hurwitz space is realized here by a projective line with the projective parameter $s$.
(In pullbacked Fuchsian equations, $s$ will be the isomonodromy parameter.)
One can check that
\begin{equation} \label{eq:phi81}
\varphi_8(x)-1=
\frac{\left(3(s+1)^2x^4-4s(s+3)x^3+12s(s-1)x^2+24s^2x+8s^2\right)^2}{64\,s\,\left(x^2-2sx-s\right)^3}.
\end{equation}
It is evident that the ramification pattern is indeed (\ref{eq:ramif8}). 
The extra ramification point is $x=5s$. 
To get a properly normalized expression, the degree 3 polynomial in the numerator of $\varphi_8(x)$
has to be factorized. We reparametrize 
\begin{equation} \label{eq:ts8}
s=\frac{2(u-1)}{u^3+4u^2+2u+2}
\end{equation}
and make the fractional-linear transformation
\begin{equation} \label{eq:lambda8}
x\mapsto \frac{2(u+8)w}{3(u+2)^4}(2x-1)-\frac{2(u-1)(u^2-4u-24)}{3(u+2)^4},
\end{equation}
where $w=\sqrt{u(u-1)(u+3)(u+8)}$. Apparently, the Hurwitz space parametrising the
properly normalized almost Belyi covering has genus 1. 
We obtain the following properly normalized expression:
\begin{eqnarray} \label{eq:phi8a}
\hspace{-18pt} \widehat{\varphi}_8(x)=
\frac{u^5(u+8)^3(u+3)}{8(u^3+4u^2+2u+2)} 
\frac{x\,(x-1)\,(x-t_8)\!\left(x\!-\!\frac12\!-\!\frac{(u-1)(u^2-4u-24)}{2w(u+8)}\right)^5}
{\left(x^2-(L_1w+1)\,x+\frac{1}2L_1w-L_2\right)^3},
\end{eqnarray}
where
\begin{equation} \label{eq:t8}
t_8=\frac12+\frac{(u-1)(3u^4+12u^3+24u^2+64u+32)}{2u^2(u+8)\sqrt{u(u-1)(u+3)(u+8)}},
\end{equation}
and
\begin{eqnarray*} \textstyle
L_1=\frac{(u-1)(u+4)(u^2-10)}{(u+3)(u+8)^2(u^3+4u^2+2u+2)},\qquad
L_2=\frac{5u^6+40u^5+20u^4-320u^3-40u^2+1216u-192}{8(u+3)(u+8)^2(u^3+4u^2+2u+2)}.
\end{eqnarray*}
To get to the degree 8 covering in \cite[pages 11--12]{HGBAA}, 
one has to make the substitutions $u\mapsto -8(s+1)^2/(s^2-34s+1)$ or $u\mapsto (8s_1+1)/(1-s_1)$ . 
After the first substitution, the quadratic polynomial in the denominator of 
(\ref{eq:phi8a}) factors as well.

The degree 12 covering is given by:
\begin{equation} \label{eq:phi12}
\varphi_{12}(x)= \frac4{27(s+4)^3}\frac{F_{12}^3}{x^5G_{12}}, \quad\mbox{or}\quad
\varphi_{12}(x)-1= \frac1{27(s+4)^3}\frac{P_{12}^2}{x^5G_{12}},
\end{equation}
where
\begin{eqnarray} \label{eq:ph12p}
F_{12}\!\equal x^4-4(s+3)x^3+(s^2+6s+14)x^2+2(s+6)x+1, \nonumber\\
G_{12}\!\equal sx^3-4(s^2+3s-1)x^2-4(2s+11)x-4,\\  \nonumber
P_{12}\!\equal 2x^6-12(s+3)x^5+15(s^2+6s+10)x^4+2s(s^2+9s+15)x^3\qquad\\
&& +6(s^2+9s+25)x^2+6(s+6)x+2. \nonumber
\end{eqnarray}
The extra ramification point is $x=-5/s$. 
To get a properly normalized expression, we reparametrize 
\begin{equation} \label{eq:ts12}
s=\frac{(u^2-5)(u^2+4u-1)(u^2-4u-1)}{8(u+1)^2(u-1)^2},
\end{equation}
and make the fractional-linear transformation
\begin{equation} \label{eq:lambd12}
x\mapsto \frac{(u+1)^2(u-1)^2}{2(u^2-5)}-\frac{(u+1)^3(u-3)^3(u^2+3)\,x}{2(u-1)^2(u^2-5)(u^2-4u-1)}.
\end{equation}
The obtained expression is
\begin{eqnarray}
\!\!\!\widehat{\varphi}_{12}(x)=\!\frac{1024(u+1)^{20}(u-3)^{12}
\left(x^4-\frac{(u^2-4u-1)(3u^6-21u^4+49u^2+33)}{(u+1)^5(u-3)^3}x^3+L_6\right)^3}
{27(u^2\!+\!3)^5(u^2\!-\!5)^5(u^2\!+\!4u\!-\!1)(u^2\!-\!4u\!-\!1)^5
\,x(x-1)\!\left(x-t_{12}\right)\!\left(x-t^*_{12}\right)^5},
\end{eqnarray}
where
\begin{eqnarray} \label{eq:t12}
t_{12}= \frac{(u-1)^5(u+3)^3(u^2-4u-1)}{(u+1)^5(u-3)^3(u^2+4u-1)},
\qquad t^*_{12}=\frac{(u-1)^4(u^2-4u-1)}{(u+1)(u-3)^3(u^2+3)},
\end{eqnarray}
and
\begin{eqnarray*}
L_6\equal \textstyle
\frac{(u^2-4u-1)^2(49u^{12}-686u^{10}+3895u^8-9700u^6+10575u^4-2446u^2+2409)}
{16(u+1)^{10}(u-3)^6}x^2 \\&& \textstyle
-\frac{(u-1)^5(u^2-4u-1)^3(9u^8-144u^6+874u^4-2184u^2+2469)}{8(u+1)^{10}(u-3)^9}x
+\frac{(u-1)^{10}(u+3)^2(u^2-4u-1)^4}{16(u+1)^{10}(u-3)^{10}}.
\end{eqnarray*}
The Hurwitz space parametrising this
properly normalized almost Belyi covering has still genus 0. 
To get the degree 12 covering in \cite{HGBAA}, one has to consider $1\big/\widehat{\varphi}_{12}(x)$,
and substitute   $u\mapsto (s-3)/(s+1)$.

In \cite{HGBAA}, the following symbol is introduced to denote 
$RS$-pullback transformations of $E(e_0,e_1,0,e_\infty;t;z)$ with respect to a covering with ramification
pattern $R_4(P_0|P_1|P_\infty)$: 
\begin{equation}
RS^2_4\left( \,e_0\, \atop \,P_0\, \right| 
{\,e_1\, \atop \,P_1\, } \left| \,e_\infty\, \atop \,P_\infty\, \right),
\end{equation}
where the subscripts 2 and 4 indicate a second order Fuchsian system with 4 singular points 
after the pullback. 
We assume the same assignment of the fibers $z=0$, $z=1$, $z=\infty$ as for the $R_4$-notation.
Location of the $x$-branches $0,1,t,\infty$ does not have to be normalized.
In Section \ref{sec:fullrs}, we present explicit computations for 
$RS^2_4\left( 1/5 \atop 5+1+1+1 \right| {1/2 \atop 2+2+2+2} \left| 1/3 \atop 3+3+2 \right)$ and
$RS^2_4\left( 2/5 \atop 5+1+1+1 \right| {1/2 \atop 2+2+2+2} \left| 1/3 \atop 3+3+2 \right)$.
These $RS$-pullbacks 
produce algebraic solutions of $P_{VI}(1/5,1/5,1/5,\pm1/3;t)$ 
respectively $P_{VI}(2/5,2/5,2/5,\pm2/3;t)$. 

As was noticed in \cite{HGBAA} and \cite{Do}, some
algebraic Painlev\'e VI solutions determined by $RS$-pullback transformations 
$RS^2_4\left( \,1/k_0\, \atop P_0 \right| {\,1/k_1\, \atop P_1 } \left| \,1/k_\infty\, \atop P_\infty \right)$,
with $k_0,k_1,k_\infty\in\ZZ$, can be calculated from the rational covering alone, without computing any Schlesinger transformation. Here is a general formulation of this situation.
\begin{theorem} \label{kit:method}
Let $k_0,k_1,k_\infty$ denote three integers, all $\ge 2$. Let $\varphi:\PP^1_x\to\PP^1_z$ denote
an almost Belyi map, dependent on a parameter $t$. Suppose that the following conditions are satisfied:
\begin{itemize}
\item[(i)] The covering $z=\varphi(x)$ is ramified above the points $z=0$, $z=1$, $z=\infty$; there is
one simply ramified point $x=y$ above $\PP_z^1\setminus\{0,1,\infty\}$; 
and there are no other ramified points.
\item[(ii)] The points $x=0$, $x=1$, $x=\infty$, $x=t$ 
lie above the set $\{0,1,\infty\}\subset\PP^1_z$.
 \item[(iii)] The points in $\varphi^{-1}(0)\setminus\{0,1,t,\infty\}$ are all ramified with the order $k_0$.
The points in $\varphi^{-1}(1)\setminus\{0,1,t,\infty\}$ are all ramified with the order $k_1$.
The points in $\varphi^{-1}(\infty)\setminus\{0,1,t,\infty\}$ are all ramified with the order $k_\infty$.
\end{itemize}
Let $a_0,a_1,a_t,a_{\infty}$ denote the ramification orders at $x=0,1,t,\infty$, respectively. 
Then the point $x=y$, as a function of $x=t$, is an algebraic solution of 
\begin{equation} \label{eq:p6kit}
P_{VI}\left(\frac{a_0}{k_{\varphi(0)}},\frac{a_1}{k_{\varphi(1)}},
\frac{a_t}{k_{\varphi(t)}}, 1-\frac{a_{\infty}}{k_{\varphi(\infty)}};t\right).
\end{equation}
\end{theorem}
\begin{proof} See Theorem 3.1 in \cite{KV3}.
\end{proof}

Our two coverings $\widehat{\varphi}_8(x)$, $\widehat{\varphi}_{12}(x)$ immediately give solutions of 
$P_{VI}(1/5,1/5,1/5,1/3;t)$, $P_{VI}(1/5,1/5,1/5,1/5;t)$, respectively. To parametrize the algebraic solutions, it is convenient to parametrize 
the indeterminant $t$ as, respectively,  
$t_8$ in (\ref{eq:t8}) or $t_{12}$ in (\ref{eq:t12}).  

Direct application of Theorem \ref{kit:method} to $\widehat{\varphi}_8(x)$ gives the 
following\footnote{Throughout this paper, the indices 26,  27, 31, 32 
refer to the Boalch types of icosahedral Painlev\'e VI solutions.}
solution $y_{26}(t_{8})$ of  $P_{VI}(1/5,1/5,1/5,1/3;t_{8})$:
\begin{eqnarray} \label{eq:y26}
y_{26}\equal \frac12+\frac{(u-1)(u+3)(u^3+4u^2+14u+8)}{2(u^3+4u^2+2u+2)\sqrt{u(u-1)(u+3)(u+8)}}.
\end{eqnarray}
Note that $t_{8}$ is the non-ramified point of $\widehat{\varphi}_{8}(x)$ above $z=0$
not equal to $x=0$ or $x=1$, while $y_{26}$ is the extra ramification $x$-point of 
$\widehat{\varphi}_{8}(x)$; it corresponds to the point $x=5s$
in the expression (\ref{eq:phi8}) of ${\varphi}_{8}(x)$.
To get the parametrizations in \cite{HGBAA}, one has to substitute  
$u\mapsto -8(s+1)^2/(s^2-34s+1)$ or $u\mapsto (8s_1+1)/(1-s_1)$.
In Section \ref{sec:fullrs}, we derive the same algebraic solution 
by computing the full transformation 
$RS^2_4\left( 1/5 \atop 5+1+1+1 \right| {1/2 \atop 2+2+2+2} \left| 1/3 \atop 3+3+2 \right)$.

Similarly, application of Theorem \ref{kit:method} to $\widehat{\varphi}_{12}(x)$ gives
the following solution $y_{31}(t_{12})$ of  $P_{VI}(1/5,1/5,1/5,1/5;t_{12})$:
\begin{eqnarray} \label{eq:yt31}
y_{31}=\frac{(u-1)^4(u+3)^2}{(u-3)(u+1)(u^2+3)(u^2+4u+1)}.
\end{eqnarray}
To get the parametrization in \cite{HGBAA}, one has to substitute  
$u\mapsto (s-3)/(s+1)$. The implied $RS$-transformation is
$RS^2_4\left( 1/3 \atop 3+3+3+3 \right|  {1/2 \atop 2+2+2+2+2+2} \left| 1/5 \atop 5+4+1+1+1 \right)$.
As Section \ref{sec:algsols} will demonstrate, Theorem \ref{kit:method} can be applied to an alternative normalization of ${\varphi}_{12}(x)$, giving a solution of $P_{VI}(1/4,1/4,1/4,-1/4;t)$.

Notice that the genus of algebraic Painlev\'e VI solutions is not a monotonic function of the minimal genus of Hurwitz spaces parametrizing the pull-back covering: the degree 8 covering
$\widehat{\varphi}_8(x)$  gives a genus 1 solution, while the degree 12 covering
 $\widehat{\varphi}_{12}(x)$ gives a genus 0 solution. Notice that the covering $\varphi_8(x)$ is still parametrized by a projective line, even if its normalization $\widehat{\varphi}_8(x)$ gives an algebraic Painlev\'e VI solution of genus 1. 

\section{Computation of Schlesinger transformations} 
\label{sec:fullrs}

This section starts with construction of
$RS^2_4\left( 1/5 \atop 5+1+1+1 \right| {1/2 \atop 2+2+2+2} \left| 1/3 \atop 3+3+2 \right)$,
demonstrating construction of the $S$-part of full $RS$-pullbacks as a single Schlesinger
transformation. This will gives us the same Painlev\'e VI solution $y_{26}(t_{8})$ as dictated
by Theorem \ref{kit:method}. From the full $RS$-transformations, we also easily derive a solution
of $P_{VI}(1/5,1/5,1/5,-1/3;t_8)$. 
Then we construct an example of
$RS^2_4\left( 2/5 \atop 5+1+1+1 \right| {1/2 \atop 2+2+2+2} \left| 1/3 \atop 3+3+2 \right)$
and derive solutions of $P_{VI}(2/5,2/5,2/5,2/3;t_8)$ and $P_{VI}(2/5,2/5,2/5,-2/3;t_8)$ of Boalch type 27.

Application of Appendix formulas (\ref{eq:asymtz0})--(\ref{eq:solinf}) to the equation 
$E(1/5,1/2,0,1/3;t;z)$ yields the following leading terms of dominant local solutions at the singular points, up to multiplication by constants: 
\begin{equation} \label{eq:leadus}
u_0={11\choose 1}z^{-\frac1{10}},\qquad u_1={11\choose -19}(1-z)^{-\frac14},
\qquad u_\infty={1\choose 0}z^{\frac1{6}}.
\end{equation}
Let $f_1(z), f_2(z)$ denote the normalized basis for solutions of $E(1/5,1/2,0,1/3;t;z)$. We have
$f_1(z)\sim {1\choose 0}z^{1/6}$ and $f_2(z)\sim {0\choose 1}z^{-1/6}$, as $z\to\infty$.
Up to scalar multiples, explicit expressions for these solutions can be copied from 
(\ref{eq:lsolzi1})--(\ref{eq:lsolzi2}).

The Fuchsian system 
for the equation $P_{VI}(1/5,1/5,1/5,1/3;t)$ must be an $RS$-pullback 
$RS^2_4\left( 1/5 \atop 5+1+1+1 \right| {1/2 \atop 2+2+2+2} \left| 1/3 \atop 3+3+2 \right)$
with respect to the covering $z=\widehat{\varphi}_{8}(x)$.
It is preferable to work with less the elaborate parametrization $z={\varphi}_{8}(x)$,
and apply the fractional-linear transformation (\ref{eq:lambda8}) to switch to
$z=\widehat{\varphi}_{8}(x)$ at the last stage. Let us denote
\begin{eqnarray}
F_8 \equal 9(s+1)^2x^3-24s(s+3)x^2+8s(11s-1)x+48s^2, \nonumber\\
P_8 \equal 3(s+1)^2x^4-4s(s+3)x^3+12s(s-1)x^2+24s^2x+8s^2,\\
G_8 \equal x^2-2sx-s,\nonumber
\end{eqnarray}
so that, copying (\ref{eq:phi8}) and (\ref{eq:phi81}), we have
\begin{equation}
\varphi_8(x)=\frac{(s+1)^2}{64s}\,\frac{x^5F_8}{G_8^3},\qquad\qquad
\varphi_8(x)-1=\frac{1}{64s}\,\frac{P^2_8}{G_8^3}.
\end{equation}
The direct pullback of $E(1/5,1/2,0,1/3;t,z)$ with respect
to $\varphi_{8}(x)$ is a Fuchsian system with singularities at $x=\infty$ and the roots of $F_{8}(x)$, and apparent singularities at $x=0$ and the roots of $G_{8}(x)$, $P_{8}(x)$. In particular, the local monodromy exponents at $x=\infty$ are $\pm1/3$, twice the exponents at $z=\infty$. 
We have to remove apparent singularities, and choose a solution basis $g_1(x)$, $g_2(x)$
of the pull-backed equation so that, up to constant multiples,  $g_1(x)\sim{1\choose 0}x^{1/10}$ and 
$g_2(x)\sim{0\choose1} x^{-1/10}$. This would allow straightforward 
normalization\footnote{We may require strict asymptotic behavior for $g_1(x)$, $g_2(x)$ without reference to constant multiples, as in \cite{AK2} and \cite{HGBAA}, but this is unnecessary. The Jimbo-Miwa correspondence merely requires existence of a basis with the strict asymptotics. There is no value in controlling strict identification of normalized bases all the way until final fractional-linear normalization, (\ref{eq:lambda8}) in this particular case.}
of the pullbacked equation for the Jimbo-Miwa correspondence. 

Let $T_{26}$ denote the matrix representing the basis $g_1(x)$, $g_2(x)$ in terms of the solution
basis  $f_1(\varphi_8(x))$, $f_2(\varphi_8(x))$ of the directly pullbacked equation.
That is, ${g_1\choose g_2}=T_{26}{f_1\choose f_2}$. 
The $S$-matrix in (\ref{eq:rstrans})--(\ref{eq:pbacked}) can be taken to be $T_{26}^{-1}$. 
It has to shift local exponents at $x=0$, and the roots of $G_8(x)$ and $P_8(x)$. The local exponents
at $x=\infty$ have to be shifted  as well, since the shifts of local monodromy differences must add to
an even integer. The matrix $T_{26}$ has to satisfy the following conditions:
\begin{enumerate}
\item {\bf Local exponent shifts for general vectors.}
For general vectors $u$, the vector $T_{26}u$ is:
$O\left(1/\sqrt{x}\right)$ at $x=0$; 
$O\left(1/\sqrt{P_{8}}\right)$ at the roots of $P_{8}(x)$;
$O\left(1/\sqrt{G_{8}}\right)$ at the roots of $G_{8}(x)$;
and $O(\sqrt{x})$ at infinity.
 \item {\bf Local exponent shifts for dominant solutions at singular points.}
We must have $T_{26}u_0=O(\sqrt{x})$ at $x=0$; $T_{26}u_1=O(\sqrt{P_8})$ at the roots of $P_8(x)$;
$T_{26}u_{\infty}=O(\sqrt{G_{8}})$ at the roots of $G_{8}(x)$; 
and $T_{26}u_\infty=O\left(1/\sqrt{x}\right)$ 
at infinity.
 \item {\bf Normalization at infinity.} The positive local monodromy exponent $1/{6}$ at $z=\infty$ gets
transformed to the local monodromy exponent $2\cdot\frac1{6}-\frac12<0$ at $x=\infty$.   
Hence the dominant solution $\sim {1\choose0}z^{1/6}$ should be mapped, up to a constant multiple,
to the vanishing solution $\sim {0\choose1}x^{-1/6}$; 
and the vanishing solution $\sim {0\choose1}z^{-1/6}$ should be mapped,
up to a constant multiple, to the dominant solution $\sim {1\choose0}x^{1/6}$. 
\end{enumerate}
By the first condition, the matrix $T_{26}$ has the form
\begin{equation}
T_{26}=\frac{1}{\sqrt{x\,G_{8}P_{8}}}\, \left(\begin{array}{cc} A_{26} & B_{26}\\ C_{26} & D_{26}
\end{array}\right),
\end{equation}
where the matrix entries $A_{26}$, $B_{26}$, $C_{26}$, $D_{26}$ are polynomials in $x$
of maximal degree 4, with the coefficients being rational functions in $s$. By the second
condition, the expressions $11A_{26}+B_{26}$ and $11C_{26}+D_{26}$ vanish at $x=0$;
$11A_{26}-19B_{26}$ and $11C_{26}-19D_{26}$
are divisible by $P_{8}$; and $A_{26}$, $C_{26}$ are divisible by $G_{8}$ and have
degree at most 3. Let us denote a few polynomials:
\begin{eqnarray*}
U_1=19\cdot\frac{11A_{26}+B_{26}}{x},&& V_1=\frac{11A_{26}-19B_{26}}{P_{8}},\qquad
 W_1=-220\cdot\frac{A_{26}}{G_{8}},\\
U_2=19\cdot\frac{11C_{26}+D_{26}}{x},&& V_2=\frac{11C_{26}-19D_{26}}{P_{8}},\qquad
 W_2=-220\cdot\frac{C_{26}}{G_{8}}.
\end{eqnarray*}
Then for $i=1,2$ we have:
\begin{equation} \label{eq:syzrel}
x\,U_i+P_{8}\,V_i+G_{8}\,W_i=0.
\end{equation}
In other words, the two polynomial vectors $(U_i,V_i,W_i)$ are syzygies
between the three polynomials $x$, $P_{8}$, $G_{8}$.
The last condition sets up the degrees for the entries of $T_{26}$:
\begin{equation} \label{eq:syzdeg26} 
\deg A_{26} \le 2, \quad \deg B_{26} = 4, \quad \deg C_{26} = 3,\quad \deg D_{26} \le 3.
\end{equation}

As it turns out, the syzygies giving relations (\ref{eq:syzrel}) of degree at most 4 form a
linear space of dimension 3. Here is a basis:
\begin{eqnarray} \label{eq:syzsols26}  \textstyle
(G_{8},0,-x), \qquad (xG_{8},0,-x^2),\qquad
\left( L_1,\,-1,\, -8s \right),
\end{eqnarray}
where $L_1=3(s+1)^2x^3-4s(s+3)x^2+4s(3s-1)x+8s^2$.  The third syzygy 
gives the entries $A_{26}$, $B_{26}$ satisfying (\ref{eq:syzdeg26}).
The first syzygy in (\ref{eq:syzsols26}) gives the entries $C_{26}$, $D_{26}$.
For constructing a transformation matrix $T_{26}$, we may multiply the syzygies (or the rows) 
by constant factors. Here is a suitable transformation matrix: 
\begin{eqnarray}
T_{26}=\frac{1}{\sqrt{x\,G_{8}P_{8}}} \left( \begin{array}{cc}
38s\,G_{8} & 55P_{8}+22sG_8 \\ 
19\,x\,G_{8} & 11\,x\,G_{8} \end{array} \right).
\end{eqnarray}
Using (\ref{eq:pbacked}) with $S=T_{38}^{-1}$, we routinely compute the transformed differential equation:
\begin{eqnarray} \label{eq:transff26}
\frac{d\Psi}{dx}=\frac{1}{F_{8}} \left( \begin{array}{cc}
K_1 & -\frac85s\,K_2 \\
-\frac25(x-5s) & -K_1 \end{array} \right)\Psi,
\end{eqnarray}
where
\begin{eqnarray*} \textstyle
K_1= \frac32(s+1)^2x^2-2s(s+3)x-\frac45s,\qquad
K_2=(15s^2+30s+4)x-7s(5s+4).
\end{eqnarray*}
Notice that the $x$-root of the lower-left entry of the transformed equation is the extra ramification point
of the covering $z=\varphi_8(x)$. We can apply the Jimbo-Miwa correspondence after reparametrization (\ref{eq:ts8}) and fractional-linear transformation (\ref{eq:lambda8}) 
of (\ref{eq:transff26}). Then a solution of $P_{VI}(1/5,1/5,1/5,1/3;t)$ is equal to
the $x$-root of the lower-left entry, while the independent variable $t$ is parametrized
by the singularity $t_8$ of the transformed Fuchsian equation. 
We get the same solution $y_{26}(t_{8})$ as in (\ref{eq:y26}), (\ref{eq:t8}),
reaffirming Theorem \ref{kit:method} for this case.
 
The full $RS$-pullback 
$RS^2_4\left( 1/5 \atop 5+1+1+1 \right| {1/2 \atop 2+2+2+2} \left| 1/3 \atop 3+3+2 \right)$
can provide a lot more additional results. For instance, the $x$-root of upper-right entry of $T_{26}$ determines a solution of $P_{VI}(1/5,1/5,1/5,-1/3;t_8)$. After applying transformations 
(\ref{eq:ts8})--(\ref{eq:lambda8}) to $K_2$, the $x$-root gives the following solution
$\widetilde{y}_{26}(t_{8})$: 
$$
\widetilde{y}_{26}=\frac12+
\frac{(u-1)(u+3)(2u^6+28u^5+106u^4+169u^3+274u^2+142u+8)}
{4(u^6+8u^5+35u^4+65u^3+5u^2-22u-11)\sqrt{u(u-1)(u+3)(u+8)}}.
$$
Alternatively, this solution can be computed from $y_{26}(t_{8})$ by applying a few Okamoto transformations.

Now we consider construction of an $RS$-pullback 
$RS^2_4\left( 2/5 \atop 5+1+1+1 \right| {1/2 \atop 2+2+2+2} \left| 1/3 \atop 3+3+2 \right)$
with respect to $\widehat{\varphi}_{8}(x)$, aiming for a solution of $P_{VI}(2/5,2/5,2/5,2/3;t)$. 
The leading terms of dominant local solutions of $E(2/5,1/2,0,1/3;t;z)$ at the singular points 
are constant multiples of
\begin{equation} \label{eq:leadvs}
 v_0={17\choose 7}z^{\frac15},\qquad 
v_1={17\choose -13}(1-z)^{-\frac14},\qquad
v_\infty={1\choose 0}z^{-\frac16}.
\end{equation}
Again, it is preferable to work first with the less elaborate covering $z={\varphi}_{8}(x)$.
The direct pullback of $E(2/5,1/2,0,1/3;t,z)$ with respect to $\varphi_{8}(x)$ is a Fuchsian system with the same singularities as in the previous case, but the local monodromy exponents at $x=0$ and the roots of $F_{8}(x)$ are multiplied by 2. Hence we have to shift the local exponent difference at $x=0$ by $2$, and we do not shift the local exponents at $x=\infty$.
Let $T_{27}$ denote the transition matrix to a basis of Fuchsian solution normalized at $x=\infty$, 
analogous to $T_{26}$ above. The matrix $T_{27}$ has to satisfy the following conditions:
\begin{enumerate}
\item {\bf Local exponent shifts for general vectors.}
For general vectors $u$, the vector $T_{27}u$ is:
$O\left(1/x\right)$ at $x=0$;
$O\left(1/\sqrt{P_{8}}\right)$ at the roots of $P_{8}$;
$O\left(1/\sqrt{G_{8}}\right)$ at the roots of $G_{8}$; and $O(1)$ at
infinity. Hence, the matrix $T_{27}$ has the form
\begin{equation}
T_{27}=\frac{1}{x\,\sqrt{G_8P_8}}\, \left(\begin{array}{cc} A_{27}& B_{27}\\ C_{27} & D_{27}
\end{array}\right),
\end{equation}
where $A_{27}$, $B_{27}$, $C_{27}$, $D_{27}$ are polynomials in $x$ 
of maximal degree 4. 
 \item {\bf Local exponent shifts for dominant solutions at singular points.}
We must have: $T_{27}v_0=O(x)$ at $x=0$; $M_{27}v_1=O(\sqrt{P_{8}})$ at the roots of $P_{8}(x)$; 
and $M_{27}v_{\infty}=O(\sqrt{G_8})$ at the roots of $G_{8}(x)$.
This means that the following are triples of polynomials in $x$:
\begin{eqnarray*}
\left( 13\cdot\frac{17A_{27}+7B_{27}}{x^2}, \; 7\cdot\frac{17A_{27}-13B_{27}}{P_{8}}, \;
 -340\cdot\frac{A_{27}}{G_8} \right),\\
\left( 13\cdot\frac{17C_{27}+7D_{27}}{x^2}, \; 7\cdot\frac{17C_{27}-13D_{27}}{P_{8}}, \;
 -340\cdot\frac{C_{27}}{G_8} \right),
\end{eqnarray*}
and the polynomial triples are syzygies between $x^2$, $P_{8}$, $G_8$. 
\item {\bf Normalization at infinity.} The local exponents at $x=\infty$ are not shifted by the Schlesinger transformation. Hence the dominant solution $\sim {1\choose0}z^{1/6}$ is mapped,
up to a constant multiple, to the dominant solution $\sim {1\choose0}x^{1/3}$; 
and the vanishing solution $\sim {0\choose1}z^{-1/6}$ is mapped,
up to a constant multiple, to the vanishing solution $\sim {0\choose1}x^{-1/3}$. 
This sets up the degrees for the entries of $T_{27}$:
\begin{equation} \label{eq:syzd27}
\deg A_{27}=4, \quad \deg B_{27} \le 3, \quad \deg C_{27} \le 3,\quad \deg D_{27} \le 4.
\end{equation}
\end{enumerate}
As it turns out, the syzygies relations 
of degree at most 4 form a linear space of dimension 2. Here is a syzygy basis:
\begin{eqnarray}   \label{eq:syzs27} 
S_1=\left( G_8,0,-x^2 \right), \quad
S_2=\left( (s+1)\left(3(s+1)x^2-4sx-4s\right),\, -1,\, -8s\,(x+1) \right).
\end{eqnarray}
To determine the entries $C_{27}$, $D_{27}$, we may take the syzygy $S_2$. 
To determine the entries $A_{27}$, $B_{27}$, we may take the syzygy 
$60(s+1)^2S_1-7S_2$. 
Up to multiplication of the two rows by scalar factors, we obtain 
\begin{eqnarray*}
&  A_{27}=-\frac{1}{17}\left(15(s+1)^2x^2-14sx-14s\right)G_8, & 
\textstyle B_{27}=\frac{17}{13}A_{27}+\frac{5}{13}P_8,\\
& C_{27}=\frac{14}{17}\,s\,(x+1)\,G_8, & 
\textstyle D_{27}=\frac{17}{13}C_{27}+\frac{5}{13}P_8.
\end{eqnarray*}
The transformed differential equation is
\begin{eqnarray}
\frac{d\Psi}{dz}=\frac{1}{F_8} \left( \begin{array}{cc}
K_3 & -2sK_4\\
\frac{14}5s\left(x+\frac{8(5s-1)}{15(s+1)}\right) & -K_3 \end{array} \right)\Psi.
\end{eqnarray}
where 
\begin{eqnarray*}
K_3\equal 
3(s+1)^2x^2-\frac{2s(25s+54)}5x+\frac{8s(150s^2+235s+1)}{75(s+1)},\\
K_4\equal (15s^2+20s-8)x+\frac{16(75s^2+100s+4)}{75(s+1)}.
\end{eqnarray*}
To get a solution of $P_{VI}(2/5,2/5,2/5,2/3;t)$ by the Jimbo-Miwa correspondence,
we have to apply reparametrization (\ref{eq:ts8}) and fractional-linear transformation (\ref{eq:lambda8}) 
to the lower left entry of the differential equation, and write down the $x$-root. We get the following solution $y_{27}(t_{8})$:
\begin{equation}
y_{27}=\frac12+\frac{(u+3)(4u^3-7u^2+4u+8)}{10u\sqrt{u(u-1)(u+3)(u+8)}}.
\end{equation}
To get the same parametrization of this solution as in \cite{Bo2}, one has to substitute
$u\mapsto-6s/(2s+1)$.

In the same way, the $x$-root of upper right entry $-3sK_4$ determines a solution 
of the equation $P_{VI}(2/5,2/5,2/5,-2/3;t)$. The solution $\widetilde{y}_{27}(t_{8})$ is the following:
$$
\widetilde{y}_{27}=\frac12+\frac{(u\!+\!3)(8u^9\!+\!90u^8\!+\!216u^7\!+\!670u^6
\!+\!2098u^5\!-\!571u^4\!-\!850u^3\!-\!7u^2\!-\!140u\!-\!56)}
{50u(2u^6\!+\!16u^5\!+\!30u^4\!+\!10u^3\!+\!45u^2\!+\!46u\!+\!13)\sqrt{u(u\!-\!1)(u\!+\!3)(u\!+\!8)}}.
$$

\section{Syzygies for $RS$-pullback transformations}
\label{sec:generalrs}

As we saw in the previous section, computation of Schlesinger transformations for full $RS$-pullback
transformations leads to computation of syzygies between three polynomials in one variable $x$. 
Recently, this syzygy problem got a lot of attention in computational algebraic geometry of rational
curves \cite{Cox}, \cite{CSC}. It was successfully considered by Franz Meyer \cite{fmeyer} already in 1887. David Hilbert famously extended Meyer's results in \cite{hilbert}.

Here are basic facts regarding the homogeneous version of the syzygy problem.
\begin{theorem}
Let $\fK$ denote a field, and let $n$ denote an integer.  Let $P(u,v)$, $Q(u,v)$, $R(u,v)$ denote homogeneous polynomials in $\fK[u,v]$ of degree $n$. We assume that these polynomials have no common factors. Let $Z$ denote the graded $\fK[u,v]$-module
of syzygies between $P(u,v)$, $Q(u,v)$, $R(u,v)$.

The module $Z$ is free of rank $2$. If $(p_1,q_1,r_1)$, $(p_2,q_2,r_2)$ is a homogenous basis for $Z$, then 
\begin{equation} \label{eq:hsyzdeg}
\deg (p_1,q_1,r_1)+\deg (p_2,q_2,r_2)=n,
\end{equation}
and the polynomial vector $(P,Q,R)$ is a $\fK$-multiple of
\begin{equation} \label{eq:syzcp}
\left(q_1r_2-q_2r_1,\; p_2r_1-p_1r_2,\; p_1q_2-p_2q_1 \right).
\end{equation}
\end{theorem}
\begin{proof}
See \cite{CSC}, or even \cite{fmeyer}. The form (\ref{eq:syzcp}) is a special case of Hilbert-Burch 
theorem \cite[Theorem 3.2]{Eisenb}.
\end{proof}

In our situation, $\fK$ is a function field on a Hurwitz curve. For our applications, $\fK=\CC(s)$. 
But we rather consider syzygies between univariate non-homogeneous polynomials.
Here are the facts we use.
\begin{theorem} \label{th:syzpqr}
Suppose that $P(x)$, $Q(x)$, $R(x)$ are polynomials in $\fK[x]$ without common factors.
Let $Z$ denote the $\fK[x]$-module of syzygies between $P(x)$, $Q(x)$, $R(x)$. Then:
\begin{enumerate}
\item[(i)] The module $Z$ is free of rank $2$.
\item[(ii)] For any two syzygies $(p_1,q_1,r_1)$, $(p_2,q_2,r_2)$, expression $(\ref{eq:syzcp})$
is a $\fK[x]$-multiple of $(P,Q,R)$.
\item[(iii)] There exist syzygies $(p_1,q_1,r_1)$, $(p_2,q_2,r_2)$, such that expression $(\ref{eq:syzcp})$
is equal to $(P,Q,R)$.
\item[(iv)] Two syzygies $(p_1,q_1,r_1)$, $(p_2,q_2,r_2)$ form a basis for $Z$
if and only if expression $(\ref{eq:syzcp})$ is a nonzero $\fK$-multiple of $(P,Q,R)$.
\item[(v)] If syzygies $(p_1,q_1,r_1)$, $(p_2,q_2,r_2)$ form a basis for $Z$, and 
$\alpha_1$, $\beta_1$, $\gamma_1$, $\alpha_2$, $\beta_2$, $\gamma_2\in\fK$, then
\begin{equation}
\det\left(\begin{array}{cc}
\alpha_1p_1P+\beta_1q_1Q+\gamma_1r_1R & \alpha_2p_1P+\beta_2q_1Q+\gamma_2r_1R \\
\alpha_1p_2P+\beta_1q_2Q+\gamma_2r_1R & \alpha_2p_2P+\beta_2q_2Q+\gamma_2r_2R
\end{array}\right)
\end{equation}
is a $\fK$-multiple of $P\,Q\,R$. 
\end{enumerate}
\end{theorem}
\begin{proof}
Here are straightforward considerations. The module is free because $\fK[x]$ is a principal ideal domain.
The rank is determined by the exact sequence $0\to Z\to K^3\to K\to 0$ of free $\fK$-modules,
where the map $K^3\to K$ is defined by $(p,q,r)\mapsto pP+qQ+rR$. 

Statement {\em (ii)} holds because either expression  (\ref{eq:syzcp}) is the zero vector, or the
$\fK[x]$-module of simultaneous syzygies between the two triples $p_1,q_1,r_1$ and $p_2,q_2,r_2$
is free of rank 1. The triple $(P,Q,R)$ is a generator of this module (since $P,Q,R$ have no common factors), while triple (\ref{eq:syzcp}) belongs to the module.

For statement {\em (iii)}, let $D=\gcd(P,Q)$, $\widetilde{P}=P/D$ and $\widetilde{Q}=Q/D$.
Then $\widetilde{P}=P/D$ and $\widetilde{Q}=Q/D$ are coprime, and there exist polynomials $A$, $B$
such that $R=A\widetilde{P}+B\widetilde{Q}$. Then $(\widetilde{Q},-\widetilde{P},0)$ and
$(A,B,-D)$ are two required syzygies.

Assume now that (\ref{eq:syzcp}) is a nonzero scalar multiple of $(P,Q,R)$. 
Let us denote $u_1=(p_1,q_1,r_1)$ and $u_2=(p_2,q_2,r_2)$.
If $u_3=(p_3,q_3,r_3)$ is a syzygy in $Z$, then 
\begin{equation}
\det\left(\begin{array}{ccc}
p_1 & p_2 & p_3 \\ q_1 & q_2 & q_3 \\  r_1 & r_2 & r_3\end{array}\right)=0,
\end{equation}
because the syzygy condition gives a $\fK(x)$-linear relation between the rows. 
A $\fK[x]$-linear relation between the 3 syzygies is determined by the minors, say:
\begin{equation} \label{eq:syz3rel} \textstyle 
(p_2q_3-p_3q_2)u_1 
+(p_3q_1-p_1q_3)u_2 
+(p_1q_2-p_2q_1)u_3=0. 
\end{equation}
By the second part, each coefficient here is a polynomial multiple of $R$. By our assumption,
$p_1q_2-p_2q_1$ is a nonzero constant multiple of $R$. After dividing (\ref{eq:syz3rel}) by $R$,
we get an expression of $u_3$ 
as a $\fK[x]$-linear combination of $u_1$ and $u_2$, 
proving that the latter two syzygies form a basis for $Z$. 
On the other hand,  suppose that (\ref{eq:syzcp}) is equal to $(fP,fQ,fR)$, where either $f=0$ or
the degree of $f$ in $x$ is positive. In the former case, the syzygies $u_1$ and $u_2$ are linearly
dependent over $\fK(x)$, so they cannot form a basis for $Z$. In the latter case, one can see that
for any two $\fK[x]$-linear combinations of $u_1$ and $u_2$ the expression analogous to
(\ref{eq:syzcp}) is a multiple of $(fP,fQ,fR)$, so a syzygy referred to in part {\em (iii)} is not in the module generated by $u_1$, $u_2$. 

In the last claim {\em (v)}, we can eliminate the terms with $r_1R$ and $r_2R$ in all matrix entries thanks to the syzygy condition. Hence we consider, for some scalars $ \widetilde{\alpha}_1$, $\widetilde{\alpha}_2$, $\widetilde{\beta}_1$, $\widetilde{\beta}_2$,
\begin{equation}
\det\left(\begin{array}{cc}
\widetilde{\alpha}_1p_1P+\widetilde{\beta}_1q_1Q &
\widetilde{\alpha}_2p_1P+\widetilde{\beta}_2q_1Q \\
\widetilde{\alpha}_1p_2P+\widetilde{\beta}_1q_2Q &
\widetilde{\alpha}_2p_2P+\widetilde{\beta}_2q_2Q \end{array}\right)=
(\widetilde{\alpha}_1\widetilde{\beta}_2-\widetilde{\alpha}_2\widetilde{\beta}_1)
(p_1q_2-q_2p_1)PQ.
\end{equation} 
By the previous statement, $p_1q_2-q_2p_1$ is a scalar multiple of $R$.
\end{proof}

In the application to $RS$-transformations, we start with a matrix hypergeometric equation
$E(e_0,e_1,0,e_\infty;t;z)$ and its direct pullback with respect to a covering $z=\varphi(x)$.
After this, we have to shift local monodromy differences at some points of the fiber 
$\{0,1,\infty\}\subset\PP^1_z$. Let $k$ denote  the order of the pole $x=\infty$ 
of the rational function $\varphi(x)$, or the difference between degrees of its numerator and denominator.

Let $F(x)$ denote the polynomial whose roots are 
the points above $z=0$ where local monodromy differences have to be shifted, with the root multiplicities equal to the corresponding shifts of local monodromy differences. Let $G(x)$ and $H(x)$
denote similar polynomials whose roots are the finite points above $z=1$ respectively $z=\infty$
where  local monodromy differences have to be shifted, with corresponding multiplicities.
We set 
\begin{equation}
\Delta := \deg F(x)+\deg G(x)+\deg H(x).
\end{equation}
Suppose that the point $x=0$ is above $z=0$, and the local monodromy difference at $x=\infty$
has to be shifted by $\delta$.  The sum $\Delta+\delta$ must be even.

Local exponent shifts for general asymptotic solutions imply the following form of the inverse
Schlesinger matrix:
\begin{equation} \label{eq:gens1}
S^{-1}=\frac{1}{\sqrt{F\,G\,H}} \left(\begin{array}{cc} A & B \\ C & D
\end{array}\right).
\end{equation}
Local exponent shifts for dominant solutions at singular points require that the following are
triples of polynomials in $x$:
\begin{eqnarray}  \label{eq0:syzgen} 
& \left( \frac{(e_0+e_1-e_\infty)A+(e_0-e_1+e_\infty)B}{F}, 
\frac{(e_0+e_1-e_\infty)A-(e_1-e_0+e_\infty)B}{G}, \frac{A}{H} \right),\\  
&\label{eq0:syzgen2}  
\left( \frac{(e_0+e_1-e_\infty)C+(e_0-e_1+e_\infty)D}{F}, 
\frac{(e_0+e_1-e_\infty)C+(e_0-e_1-e_\infty)D}{G}, \frac{C}{H} \right). 
\end{eqnarray}
These polynomial triples are syzygies between
\begin{equation} \label{eq:syzgb}
(e_1-e_0+e_\infty)F, \quad (e_0-e_1+e_\infty)G, \quad -2e_\infty(e_0+e_1-e_\infty)H.
\end{equation}
More conveniently, the following polynomial triples are 
syzygies between $F$, $G$, $H$: 
\begin{eqnarray}  \label{eq:syzgen} 
& \hspace{-11pt}
 \left( \frac{\left((e_0-e_\infty)^2-e_1^2\right)A+\left((e_0-e_1)^2-e_\infty^2\right)B}{2e_\infty F}, 
\frac{\left((e_1-e_\infty)^2-e_0^2\right)A+\left(e_\infty^2-(e_0-e_1)^2\right)B}{2e_\infty G},
\frac{(e_0+e_1-e_\infty)A}{H} \right),\\  
& \hspace{-11pt} \label{eq:syzgen2}  
\left( \frac{\left((e_0-e_\infty)^2-e_1^2\right)C+\left((e_0-e_1)^2-e_\infty^2\right)D}{2e_\infty F}, 
\frac{\left((e_1-e_\infty)^2-e_0^2\right)C+\left(e_\infty^2-(e_0-e_1)^2\right)D}{2e_\infty G},
\frac{(e_0+e_1-e_\infty)C}{H} \right).
\end{eqnarray}
Normalization at infinity sets up the degrees for the entries of $S^{-1}$ if $\delta<k$,
as we show in the following lemma.

If $F$ is a Laurent polynomial or Laurent series in $1/x$,
we let $\{F\}$ denote the polynomial in $x$ part of $F$. In particular, $\{Fx^{-j}\}$ for an integer $j>0$
is equal to the polynomial quotient of the division of $F$ by $x^j$.
\begin{lemma} \label{lm:asysol}
Let $f_1(z)\sim{1\choose 0}\,z^{\frac12e_\infty}$, $f_2(z)\sim {0\choose 1}\,z^{-\frac12e_\infty}$ 
denote the normalized basis for solutions for $E(e_0,e_1,0,e_\infty;t;z)$, like in Section $\ref{sec:fullrs}$.
Suppose that the Schlesinger transformation $S$ maps $f_1(\varphi(x))$, $f_2(\varphi(x))$ to
solutions (of the pull-backed equation) asymptotically proportional to, respectively, 
\begin{equation} \label{eq:asysol}
\sim{1\choose 0}\,x^{\frac{1}2ke_\infty+\frac12\delta}, \qquad
\sim{0\choose 1}\,x^{-\frac{1}2ke_\infty-\frac12\delta}. 
\end{equation}
\begin{enumerate}
\item If $\delta=0$, we have these degree bounds for the entries of $S^{-1}$:
\begin{equation} \label{eq:syzdge0} \textstyle
\deg A=\frac{\Delta}2, \qquad \deg B<\frac{\Delta}2, \qquad
\deg C<\frac{\Delta}2,\qquad \deg D=\frac{\Delta}2.
\end{equation}
\item If $\delta>0$, let $\Delta^*=\frac12(\Delta-\delta)$, and 
$\displaystyle f_2(\varphi(x))=\theta(x) {  h_1 \choose h_2 }$, where $\theta(x)$ is a power function,
and $h_1,h_2$ are power series in $1/x$ with $h_1\to 0$, $h_2\to 1$ as $x\to\infty$.
Then \mbox{$\deg A=\frac12(\Delta+\delta)$}, the other three entries of $S^{-1}$ have lower degree, 
and 
\begin{equation} \label{eq:truncm}
\left( \begin{array}{cc} \{A\,x^{-\Delta^*-1}\} & \{B\,x^{-\Delta^*}\} \\
 \{C\,x^{-\Delta^*-2}\} & \{D\,x^{-\Delta^*-1}\} \end{array} \right)
 { \{x^{\delta}\,h_1\} \choose \{x^{\delta-1}\,h_2\} }
\end{equation}
gives a polynomial vector of degree $\le\delta-2$ in $x$.
\item If $\delta>1$, then $\deg D<\frac12(\Delta+\delta)-1$.
\item $\deg (AD-BC)\le\Delta$.
\item If $\delta<k$ then the degree bounds for the entries of $S^{-1}$ are
\begin{equation} \label{eq:syzdge} \textstyle
\deg A=\frac{\Delta+\delta}2, \qquad \deg B<\frac{\Delta-\delta}2, \qquad
\deg C<\frac{\Delta+\delta}2,\qquad \deg D=\frac{\Delta-\delta}2.
\end{equation}
\item If $\delta\le\max(2,k)$ then $\deg C<\frac{\Delta+\delta}2$ and $\deg D\le\frac{\Delta-\delta}2$.
\end{enumerate}
\end{lemma}
\begin{proof} 
The first statement is straightforward. In part {\em (ii)}, the degree bounds on $A$, $C$
follow from the action $S^{-1}$ on $f_1(\varphi(x))$, that increases the local exponent $\frac12ke_\infty$.
The prescribed action on $f_2(\varphi(x))$ should cancel the terms of $A$, $B$, $C$, $D$
of degree greater than roughly $\Delta^*$.
More presicely, that action of $S^{-1}$ 
can be explicitly written as follows:
\begin{eqnarray*}
\frac1{\sqrt{FGH}}\left( \begin{array}{cc} x^{\Delta^*} & 0 \\
0 & x^{\Delta^*+1} \end{array} \right)
\left( \begin{array}{cc} A\,x^{-\Delta^*-1} & B\,x^{-\Delta^*} \\
 C\,x^{-\Delta^*-2} & D\,x^{-\Delta^*-1} \end{array} \right)
 \left( \begin{array}{cc} x & 0 \\ 0 & 1 \end{array} \right)
  {  \theta(x)h_1 \choose  \theta(x)h_2 }\qquad \nonumber\\
=\frac{x^{2-\delta}\,\theta(x)}{\sqrt{FGH}}
\left( \begin{array}{cc} x^{\Delta^*-1} & 0 \\
0 & x^{\Delta^*} \end{array} \right)
\left( \begin{array}{cc} A\,x^{-\Delta^*-1} & B\,x^{-\Delta^*} \\
 C\,x^{-\Delta^*-2} & D\,x^{-\Delta^*-1} \end{array} \right)
  {  x^{\delta}\,h_1 \choose  x^{\delta-1} h_2 }.
\end{eqnarray*}
Entries of the last product (or a matrix and the vector) can have degree at most $\delta-2$.
The coefficients to greater powers of $x$ depend on the truncated entries in (\ref{eq:truncm}) only. 
That completes the proof of part {\em (ii)}. Note that 
\[ \begin{array}{rr}
\deg\{A\,x^{-\Delta^*-1}\}=\delta-1, & \deg\{B\,x^{-\Delta^*}\}\le\delta-1, \\
\deg\{C\,x^{-\Delta^*-2}\}\le\delta-3, & \deg\{D\,x^{-\Delta^*-1}\}\le\delta-2,\\
\deg\{x^{\delta}\,h_1\}\le\delta-k, & \deg\{x^{\delta-1}\,h_2\}=\delta-1.
\end{array}
\]
If $\delta\ge2$, then $\deg\{D\,x^{-\Delta^*-1}\}\le\delta-3$ as well, giving part {\em (iii)}.

Part {\em (iv)} is immediate if $\delta=0$.
Otherwise $\deg(AD-BC)<\Delta+\delta$ as a first estimate. 
The matrix in (\ref{eq:truncm}) is formed by the leading terms contributing to 
terms in $AD-BC$ greater than $\Delta$; its columns are linear dependent modulo 
division by $x^{\delta-1}$, and that translates to the claim of part {\em (iv)}.

If $\delta<k$, then the vector in (\ref{eq:truncm}) is simply $0\choose 1$, 
and that gives trivial restrictions  on the coefficients of $B$ and $D$ 
to the powers $\ge\Delta^*$ of $x$, giving part {\em (v}). For the last part,
we have to consider additionally $\delta\le 2$ and $\delta=k$.
If $\delta=k>2$ then $\{x^{\delta}\,h_1\}$ is a constant and the second row of the matrix in
(\ref{eq:truncm}) has degree at most $\delta-3$, hence the constant $\{x^{\delta}\,h_1\}$
does not influence the conditions on $C$ and $D$.
\end{proof}

\noindent
Explicit expressions for the solutions  $f_1(z)$, $f_2(z)$ can be obtained 
from (\ref{eq:lsolzi1})--(\ref{eq:lsolzi2}).
If the Schlesinger transformation increases the local exponent of $-\frac12ke_\infty$ by $\delta/2$,
rather than the local exponent $\frac12ke_\infty$ as in $(\ref{eq:asysol})$, then the specifications
of Lemma \ref{lm:asysol} for the diagonal entries $A$, $D$ and for the off-diagonal entries $B$, $C$ should be pairwise interchanged, and the function $f_2$ in part {\em (ii)} should be replaced by $f_1$.
If the Schlesinger transformation maps $f_1(\varphi(x))$, $f_2(\varphi(x))$ to
functions proportional to 
\[
\sim{0\choose 1}\,x^{\frac{1}2ke_\infty\pm\frac12\delta}, \qquad
\sim{1\choose 0}\,x^{-\frac{1}2ke_\infty\mp\frac12\delta},
\] respectively, 
then the rows of $S^{-1}$ must be interchanged. Normalization of the pull-backed solutions to the
$0\choose 1$ and $1\choose 0$ leading terms can be softened by allowing the leading terms of 
the transformed solutions $f_1(z)$, $f_2(z)$ to be scalar non-zero multiples of the two basis vectors.
Then the rows of $S^{-1}$ are determined up to scalar multiples, and we do not have any further conditions on the Schlesinger transformation. In particular, the rows of $S^{-1}$ can be computed independently, from syzygy $(\ref{eq:syzgen})$ or $(\ref{eq:syzgen2})$ between $F$, $G$, $H$ satisfying extra conditions of Lemma \ref{lm:asysol}. The two syzygies ought to be defined uniquely by
Lemma \ref{lm:asysol}, up to a constant multiple. (One can check that the linear problems with undetermined coefficients have one variable more than the number of linear relations from the syzygy and Lemma \ref{lm:asysol} conditions. Multiple solutions would give low degree syzygies 
$F$, $G$, $H$; generically, the two degrees in (\ref{eq:hsyzdeg}) are equal or differ just by $1$.)
\begin{theorem} \label{th:rssyzs}
\begin{enumerate}
\item The lower-left entry of the pullbacked Fuchsian equation depends only on
the syzygy $(\ref{eq:syzgen2})$ alone; that is, it does not depend on the syzygy $(\ref{eq:syzgen})$.
Similarly, the upper-right entry of the pullbacked equation depends only on the syzygy
$(\ref{eq:syzgen})$. 
\item The required syzygies $(\ref{eq:syzgen})$--$(\ref{eq:syzgen2})$ form a basis for the
$\CC(t)[x]$-module of syzygies Êbetween the polynomials $F$, $G$, $H$. 
\item The determinant $AD-BC$ is a $\CC(t)$-multiple of $F\,G\,H$.
\item The Schlesinger transformation $S$ can be assumed to have the form
\begin{equation} \label{eq:ss1form}
S=\frac{1}{\sqrt{F\,G\,H}} \left(\begin{array}{cc} D & -B \\ -C & A
\end{array}\right), \qquad
S^{-1}=\frac{1}{\sqrt{F\,G\,H}} \left(\begin{array}{cc} A & B \\ C & D
\end{array}\right),
\end{equation}
where polynomial entries $A$, $B$, $C$, $D$ are determined by syzygies
$(\ref{eq:syzgen})$--$(\ref{eq:syzgen2})$.
\end{enumerate}
\end{theorem}
\begin{proof}
The first statement can be seen directly, by checking off-diagonal entries of the matrices $S^{-1}MS$ and $S^{-1}S'$ in expression (\ref{eq:pbacked}) for the pullbacked equation. The lower-left entry is determined by the second row of $S^{-1}$ and the first columns of $MS$ and $S'$; these all depend 
$C$, $D$, but not on $A$, $B$. We have the reverse situation for the upper-right entry.

Let $(U_1,V_1,W_1)$ and $(U_2,V_2,W_2)$ denote the 2 syzygies in 
(\ref{eq:syzgen})--(\ref{eq:syzgen2}), respectively. 
The syzygies are linearly independent, since they give different degree of $A$ or $C$.
The expression $V_1W_2-V_2W_1$ is a $\CC(t)$ multiple of $(AD-BC)/GH$;
part {\em (iii)} of Lemma \ref{lm:asysol} implies that 
\mbox{$\deg(V_1W_2-V_2W_1)\le \Delta-\deg G-\deg H=\deg F$}.
We conclude that \mbox{$V_1W_2-V_2W_1$} is a $\CC(t)$ multiple of $F$ by part {\em (ii)} of 
Theorem \ref{th:syzpqr} . The two syzygies form a module basis by part {\em (iv)} 
of the same theorem.

Part {\em (iii)} follows, since $AD-BC$ is divisible by each $F$, $G$, $H$,
and has degree $\le\Delta$. We can divide one of the rows by that scalar multiple and make the determinant precisely equal to $F\,G\,H$. Then $S$ and $S^{-1}$ have the form (\ref{eq:ss1form}).
\end{proof}

\section{General expression in terms of syzygies}
\label{sec:fullrs2}

By the Jimbo-Miwa correspondence, a Painlev\'e VI solution is determined by the lower-left entry of 
a pullbacked Fuchsian system. By the third part of Theorem \ref{th:rssyzs}, that lower-left entry is determined by one syzygy (\ref{eq:syzgen2}) between $F$, $G$, $H$. In general, that syzygy depends on the first coefficients of the solution $f_2(z)\sim{0\choose 1}\,z^{-\frac12e_\infty}$. 
But if $\delta\le\max(2,k)$, we have just the degree bounds of part {\em (vi)} of Lemma \ref{lm:asysol};
then we do not need to know coefficients in the expansion of $f_2(z)$ at $z=\infty$ in order
to determine the syzygy (and eventually, the Painlev\'e VI solution). 

Taking only small shifts $\delta\le\max(2,k)$ at $x=\infty$ is enough to generate interesting solutions of the sixth Painlev\'e equation. It looks like that in this way we can generate all ``seed" algebraic solutions 
with respect to Okamoto transformations. Formula (\ref{eq:llentry}) in the following theorem is valid for any $\delta$ if only the syzygy $(U_2,V_2,W_2)$ is right; however, we specify the syzygy only if
$\delta\le\max(2,k)$. 

If $\delta>0$, we assume that the direct pullback solutions $f_1(\varphi(x))$ and $f_2(\varphi(x))$
are mapped into solutions (\ref{eq:asysol}) in the opposite order than in Lemma \ref{lm:asysol}.
The reason is that in our applications we usually apply integer shifts that change the sign of local monodromies $\pm\frac12ke_\infty$, while we wish to keep the positive local monodromy for the 
$1\choose 0$ solution. Correspondingly, if $\delta>0$ then degree bounds in parts 
{\em (v), (vi)} of Lemma \ref{lm:asysol} on the entries of $A\ B\choose C\ D$
change column-wise. In particular,  
\begin{equation} \label{eq:syzdgf}  \textstyle
\deg C=\frac{\Delta-\delta}2,\quad
\deg D<\frac{\Delta+\delta}2.
\end{equation}
\begin{theorem} \label{th:llentry}
Let $z=\varphi(x)$ denote a rational covering, 
and let $F(x)$, $G(x)$, $H(x)$ denote polynomials in $x$. 
Let $k$ denote the order of the pole of $\varphi(x)$ at $x=\infty$.
Suppose that the direct pullback of $E(e_0,e_1,0,e_\infty;t;z)$
with respect to $\varphi(x)$ is a Fuchsian equation with the following singularities:
\begin{itemize}
\item Four singularities are $x=0$, $x=1$, $x=\infty$ and $x=t$, with the local monodromy differences 
$d_0$, $d_1$, $d_t$, $d_\infty$, respectively. The point $x=\infty$ lies above $z=\infty$.
\item All other singularities in $\PP^1_x\setminus\{0,1,t,\infty\}$ are apparent singularities.
The apparent singularities above $z=0$ (respectively, above $z=1$, $z=\infty$) are 
the roots of $F(x)=0$ (respectively, of $G(x)=0$, $H(x)=0$).
Their local monodromy differences are equal to the multiplicities of those roots. 
\end{itemize}
Let us denote $\Delta=\deg F+\deg G+\deg H$, and let $\delta\le\max(2,k)$ 
denote a non-negative integer such that $\Delta+\delta$ is even.
Suppose that $(U_2,V_2,W_2)$ is a syzygy between the three polynomials $F$, $G$, $H$, 
satisfying, if $\delta=0$,
\begin{equation} \label{eq:llsd0} \textstyle
\deg U_2=\frac{\Delta}2-\deg F,\qquad \deg V_2=\frac{\Delta}2-\deg G,\qquad
\deg W_2<\frac{\Delta}2-\deg H, 
\end{equation}
or, if $\delta>0$,
\begin{equation} \label{eq:llsd1} \textstyle
\deg U_2<\frac{\Delta+\delta}2-\deg F, \quad
\deg V_2<\frac{\Delta+\delta}2-\deg G,\quad
\deg W_2=\frac{\Delta-\delta}2-\deg H.
\end{equation}
Then the numerator of the (simplified) rational function 
\begin{eqnarray} \label{eq:llentry}
\frac{U_2W_2}{G}\!\left(\! \frac{(e_0-e_1+e_\infty)}2\frac{\varphi'}{\varphi}
-\frac{(FU_2)'}{FU_2}+\frac{(HW_2)'}{HW_2} \right) 
+\frac{(e_0-e_1-e_\infty)}2\frac{V_2W_2}{F}\frac{\varphi'}{\varphi-1} \nonumber\\
+\frac{(e_0+e_1-e_\infty)}2\frac{U_2V_2}{H}\frac{\varphi'}{\varphi\,(\varphi-1)},
\end{eqnarray}
has degree $1$ in $x$, and the $x$-root of it 
is an algebraic solution of  $P_{VI}(d_0,d_1,d_t,d_\infty+\delta;t)$.
\end{theorem}
\begin{proof} 
We use a Schlesinger transformation that removes the apparent singularities and shifts
the local monodromy difference at $x=\infty$ by $\delta$.
The matrix for its inverse has the form (\ref{eq:gens1}),
with the entry degrees given by (\ref{eq:syzdge0}) or (\ref{eq:syzdgf}).
The syzygy $(U_2,V_2,W_2)$ can be identified as (\ref{eq:syzgen2}).
Let $(U_1,V_1,W_1)$ denote the syzygy in (\ref{eq:syzgen}). We have
\begin{eqnarray}
A=\frac{HW_1}{e_0+e_1-e_\infty}, & &
B=\frac{2e_\infty FU_1+(e_1-e_0+e_\infty)HW_1}{(e_0-e_1)^2-e_\infty^2},\\
C=\frac{HW_2}{e_0+e_1-e_\infty}, & &
D=\frac{2e_\infty FU_2+(e_1-e_0+e_\infty)HW_2}{(e_0-e_1)^2-e_\infty^2}.
\end{eqnarray}
Let $h$ denote the constant 
\begin{equation}
h=\frac{2e_\infty}{(e_0+e_1-e_\infty)(e_0-e_1+e_\infty)(e_0-e_1-e_\infty)}.
\end{equation}
Then
\begin{equation}
\det { A\quad B \choose C\quad D }=h\,F\,H \left(U_2W_1-U_1W_2\right).
\end{equation}
By the second part of Theorem \ref{th:rssyzs}, we may assume that this determinant is equal to
$F\,G\,H$; this would affect the lower-left entry only by a $\CC(t)$-multiple. With this assumption,
$U_2W_1-U_1W_2=G/h$.
Let us denote $K=F\,G\,H$.

Using the form (\ref{eq:ss1form}), we have
\begin{equation}
S'=\frac1{\sqrt{K}}{ \quad\! D'\quad\! -B'\, \choose -C'\quad A' }-
\frac{K'}{2\sqrt{K^3}}{ \quad\! D\quad\! -B\, \choose -C\quad A}
\end{equation}
and
\begin{equation}
S^{-1}S'=\frac1{K}{ AD'-BC' \quad\! BA'-AB'  \choose CD'-DC' \quad DA'-CB' }-
\frac{K'}{2K}{ 1\quad 0  \choose 0\quad 1 }.
\end{equation}
The lower-left entry of $S^{-1}S'$ is equal to 
\begin{equation} \label{eq:s1sd}
h\,\frac{(FU_2)'HW_2-FU_2(HW_2)'}{K}.
\end{equation}
Let $M$ denote the $2\times2$ matrix on the right-hand side of formula (\ref{eq:mhypergeom})
in the Appendix. The entries on the second row of $S^{-1}M$ are the following, from left to right:
\begin{eqnarray} 
&& \hspace{68pt} 
 \frac{2e_\infty}{\sqrt{K}}\left( FU_2+\frac{e_0-e_\infty\,\varphi}{e_0+e_1-e_\infty}HW_2\right),\\
&& \frac{2e_\infty}{(e_0-e_1+e_\infty)\sqrt{K}} \left(
\frac{2e_\infty^2\varphi-e_0^2+e_1^2-e_\infty^2}{e_0-e_1-e_\infty}FU_2-
\left(e_0+e_\infty\varphi \right)HW_2 \right).
\end{eqnarray}
The lower-left entry of the matrix $S^{-1}MS$  is the following:
\begin{eqnarray} \label{eq:lle0}
2e_\infty h\,\frac{(e_0+e_1-e_\infty)F^2U_2^2+2(e_0-e_\infty\varphi)FHU_2W_2
+(e_0-e_1-e_\infty)\,\varphi\,H^2W_2^2}{K}.
\end{eqnarray}
We can use the syzygy relation to rewrite this expression in an attractive form.
Here is a symmetric expression equal to (\ref{eq:lle0}):
\begin{eqnarray} \label{eq:lle1}
2e_\infty h\!
\left( \! (e_0\!-\!e_1\!-\!e_\infty)\frac{V_2W_2}{F}\varphi+ 
(e_0\!-\!e_1\!+\!e_\infty)\frac{U_2W_2}{G}(\varphi\!-\!1)+(e_0\!+\!e_1\!-\!e_\infty)\frac{U_2V_2}{H}\right)\!.
\end{eqnarray} 
According to (\ref{eq:pbacked}), the lower left-entry of the pullbacked Fuchsian system
is equal to (\ref{eq:lle1}) times  $\varphi'/4e_\infty\varphi(1-\varphi)$, minus (\ref{eq:s1sd}).
Up to the constant multiple $-h$, we get expression (\ref{eq:llentry}) for the lower-left entry of the pullbacked Fuchsian equation.  By the Jimbo-Miwa correspondence, this entry must determine the Painlev\'e VI solution.
\end{proof}

We give now alternative forms of expression (\ref{eq:llentry}).
Let us introduce the following notation:
\begin{equation}
f_0=\frac{e_1-e_0+e_\infty}2,\qquad
f_1=\frac{e_0-e_1+e_\infty}2,\qquad
f_\infty=\frac{e_0+e_1-e_\infty}2.
\end{equation}
Besides, for a function $\psi$ of $x$, let $\LD{\psi}$ denote the logarithmic derivative $\psi'/\psi$
 of $\psi$. The expression (\ref{eq:llentry}) can be written as follows:
\begin{equation} \label{eq:llentry2}
\frac{U_2W_2}{G}\!\left( f_1\LD{\varphi}-\LD{\frac{FU_2}{HW_2}}\right)
-f_0\,\frac{V_2W_2}{F}\LD{\varphi-1}
-f_\infty\frac{U_2V_2}{H}\LD{\frac{\varphi}{\varphi-1}}.
\end{equation}
Thanks to the syzygy relation, we have 
\begin{eqnarray*}
(FU_2)'HW_2-FU_2(HW_2)' \equal FU_2(GV_2)'-(FU_2)'GV_2\\
\equal GV_2(HW_2)'-(GV_2)'HW_2,
\end{eqnarray*}
hence these alternative expressions expressions for  (\ref{eq:llentry2}) hold:
\begin{eqnarray} 
f_1\,\frac{U_2W_2}{G}\LD{\varphi}
-\frac{V_2W_2}{F}\left( f_0\LD{\varphi-1}-\LD{\frac{GV_2}{HW_2}}\right)
-f_\infty\frac{U_2V_2}{H}\LD{\frac{\varphi}{\varphi-1}},\\
f_1\,\frac{U_2W_2}{G}\LD{\varphi}
-f_0\frac{V_2W_2}{F}\LD{\varphi-1}
-\frac{U_2V_2}{H}\left( f_\infty\LD{\frac{\varphi}{\varphi-1}}-\LD{\frac{FV_2}{GW_2}}\right).
\end{eqnarray}
Besides, these expressions for (\ref{eq:llentry2}) can be derived:
\begin{eqnarray}
\!f_1\frac{U_2W_2}{G}\!\left(\!\LD{\varphi}-\frac1{e_\infty}\!\LD{\frac{FU_2}{HW_2}}\right)\!
-f_0\frac{V_2W_2}{F}\!\left(\!\LD{\varphi\!-\!1}-\frac1{e_\infty}\!\LD{\frac{GV_2}{HW_2}}\right)\!
-f_\infty\frac{U_2V_2}{H}\!\LD{\frac{\varphi}{\varphi\!-\!1}}\!,\\
f_1\,\frac{U_2W_2}{G}\LD{\varphi}
-f_0\,\frac{V_2W_2}{F}\!\left(\!\LD{\varphi\!-\!1}-\frac1{e_1}\LD{\frac{GV_2}{HW_2}}\right)\!
-f_\infty\frac{U_2V_2}{H}\!\left(\LD{\frac{\varphi}{\varphi\!-\!1}}-\frac1{e_0}\LD{\frac{FU_2}{GV_2}}\right)\!,\\
f_1\,\frac{U_2W_2}{G}\!\left(\!\LD{\varphi}-\frac1{e_0}\LD{\frac{FU_2}{HW_2}}\right)\!
-f_0\,\frac{V_2W_2}{F}\LD{\varphi\!-\!1}
-f_\infty\frac{U_2V_2}{H}\!\left(\LD{\frac{\varphi}{\varphi\!-\!1}}-\frac1{e_0}\LD{\frac{FU_2}{GV_2}}\right)\!.
\end{eqnarray}
All these expressions are supposed to simplify greatly to a rational function with the denominator of
degree 3 in $x$, and the numerator linear in $x$. The root of the numerator determines a Painlev\'e VI solution. 

\begin{remark} \rm
If one of the components of $(U_2,V_2,W_2)$ is equal to zero, then expression
(\ref{eq:llentry}) simplifies to a single multiplicative term, and the extra ramification point of $\varphi(x)$
is a zero of the numerator. For example, if 
$(U_2,V_2,W_2)=\left(H,0,-F \right)$,
then expression (\ref{eq:llentry}) becomes 
\begin{equation} \label{eq:llentry0}
-\frac{(e_0-e_1+e_\infty)}2\,\frac{\varphi'}{\varphi}\,\frac{FH}{G}.
\end{equation}
The extra ramification point is a zero of $\varphi'(x)$. The last numerator $FH$ simplifies out if,
assuming the covering $z=\varphi(x)$ satisfies specifications of Theorem \ref{kit:method}, we have
$e_0=1/k_0$, $e_\infty=1/k_\infty$, and similarly, the last denominator $G$ simplifies out if 
$e_1=(k_1-1)/k_1$. For comparison, in the proof of Theorem \ref{kit:method} we assumed that 
$e_0=1/k_0$, $e_1=1/k_1$, $e_\infty=(k_\infty-1)/k_\infty$;
the composite Schlesinger transformation there 
corresponds to a syzygy with the third component equal to zero.

In the context of Theorem \ref{th:llentry}, there is a syzygy with a zero component satisfying 
degree specifications (\ref{eq:llsd0}) or (\ref{eq:llsd1}) if and only if one
of the following conditions holds: 
\begin{eqnarray}
\delta=0, & \deg F+\deg G=\deg H,\\
\delta>0, & \deg F+\deg H=\deg G-\delta,\\
\delta>0, & \deg G+\deg H=\deg F-\delta.
\end{eqnarray}
\end{remark}

\begin{remark} \rm
As mentioned with examples in Section \ref{sec:fullrs}, the upper-right entry of the transformed equation similarly determined a solution of another Painlev\'e VI equation. Accordingly, if we have a proper syzygy $(U_1,V_1,W_1)$ determining the upper row of the Schlesinger matrix, we may use the same formula
(\ref{eq:llentry}) with $(U_2,V_2,W_2)$ replaced by $(U_1,V_1,W_1)$ to compute a solution of
$P_{VI}(d_0,d_1,d_t,-d_\infty-\delta)$. Note that this Painlev\'e VI equation is the same as
$P_{VI}(d_0,d_1,d_t,d_\infty+\delta+2)$. Particularly, it is contiguous (and related by Okamoto transformations) to $P_{VI}(d_0,d_1,d_t,d_\infty+\delta)$.

Moreover, an algebraic solution of $P_{VI}(d_0,d_1,d_t,-d_\infty-\delta')$
can be obtained by using Theorem \ref{th:llentry} with its $\delta$ replaced by $\delta'+2$,
that is, using the lower-right entry of a pullback by contiguous Schlesinger transformation. 
It appears that the same algebraic solution is obtained regardless whether the lower-left
entry or the upper-right entry of appropriately contiguous Schlesinger transformations is used.
\end{remark}

\section{More algebraic Painlev\'e VI solutions}
\label{sec:algsols}

Here we apply Theorem \ref{th:llentry} to compute a few algebraic
Painelve VI solutions. Implicitly, we employ $RS$-transformations 
of the hypergeometric equation $E(1/3,1/2,0,2/5;t;z)$ 
with respect to the covering $z=\widehat{\varphi}_{12}(x)$
Additionally,  we note that a fractional-linear version of 
$\widehat{\varphi}_{12}(x)$ can be used to pullback $E(1/3,1/2,0,1/4;t;z)$ and $E(1/3,1/2,0,1/2;t;z)$.

The implied $RS$-pullback transformation for the equation $E(1/3,1/2,0,2/5;t;z)$ is
$RS_4^2\left(  1/3 \atop 3+3+3+3 \right| {1/2 \atop 2+2+2+2+2+2} \left| 2/5 \atop 5+4+1+1+1 \right)$.
We work mainly with the covering $z=\varphi_{12}(x)$ rather than with the normalized $z=\widehat{\varphi}_{12}(x)$, and apply reparametrization (\ref{eq:ts12}) and normalizing fractional-linear transformation (\ref{eq:lambd12}) at the latest. 
Theorem \ref{th:llentry} has to be applied with $(F,G,H)=\left(F_{12},P_{12},x^2\right)$ and $\delta=0$.
The Painlev\'e solution must solve $P_{VI}(2/5,2/5,2/5,8/5;t)$, which is the same equation
as  $P_{VI}(2/5,2/5,2/5,2/5;t)$.
Degree specifications (\ref{eq:llsd0}) are $\deg U_2=2$, $\deg V_2=0$, $\deg W_2<4$.
Up to a constant multiple, there is one syzygy satisfying these bounds: 
\begin{eqnarray} \label{eq:syzs32} \textstyle 
\left(\, x^2+(s+6)x+1, -\frac12, 
-3(s+4)\left(x^3-\left(\frac72s+11\right)x^2+(s+7)x+1\right) \, \right),
\end{eqnarray}
With this syzygy, expression (\ref{eq:llentry}) is equal to $-3(s+4)(3sx+8s+20)/10G_{12}$.
After reparametrization (\ref{eq:ts12}) and normalizing fractional-linear transformation
(\ref{eq:lambd12}), the $x$-root 
gives the following solution $y_{32}(t_{12})$ of  $P_{VI}(2/5,2/5,2/5,2/5;t_{12})$:
\begin{equation} \label{eq:y32}
y_{32}=\frac{(u-1)^2(u+3)^2(3u^2+1)}{3(u+1)^3(u-3)(u^2+4u-1)}.
\end{equation}
The solutions $y_{31}(t_{12})$ and $y_{32}(t_{12})$ are presented in \cite[Section 7]{KV1} as well,
but reparametrized $u\mapsto-(s+3)/(s-1)$. With Okamoto transformations, these two solutions can be transformed to, respectively, Great Icosahedron and Icosahedron solutions of
Dubrovin-Mazzocco \cite{DM}.

The full $RS$-transformation 
$RS_4^2\left(  1/3 \atop 3+3+3+3 \right| {1/2 \atop 2+2+2+2+2+2} \left| 2/5 \atop 5+4+1+1+1 \right)$ 
with $\delta=0$ gives us a solution of $P_{VI}(2/5,2/5,2/5,-8/5;t)$, via the upper right entry of the pullbacked Fuchsian equation. As mentioned at the end of previous section, we can use the same expression (\ref{eq:llentry}) with an appropriate syzygy $(U_1,V_1,W_1)$ for the triple 
$\left(F_{12},P_{12},x^2\right)$ to compute the Painlev\'e VI solution.
The degree constraints are the following:
\begin{equation}
\deg U_1=2,\qquad \deg V_1=0, \qquad \deg W_1=4,\qquad \deg(17U_1F_{12}+7V_1G_{12})<6.
\end{equation}
Let $S_3$ denote the syzygy $\left(-x^2,0,F_{12}\right)$.
One can take $(U_1,V_1,W_1)$ equal to syzygy (\ref{eq:syzs32}) plus $\frac{24}7S_3$.
Here is the final expression for a solution $\widehat{y}_{32}(t_{12})$ of 
$P_{VI}(2/5,2/5,2/5,-8/5;t_{12})$, obtained after application of reparametrization (\ref{eq:ts12}) and normalizing fractional-linear transformation (\ref{eq:lambd12}):
\begin{equation}
\widehat{y}_{32}=\frac{(u-1)^2(u+3)^2(13u^4-2u^2+5)(9u^6-55u^4+195u^2+299)}
{13(u+1)^3(u-3)(u^2+3)(u^2+4u-1)(9u^6-47u^4+499u^2+115)}.
\end{equation}

It is instructive to observe that to get a solution of $P_{VI}(2/5,2/5,2/5,-2/5;t_{12})$ we have to consider a Schlesinger transformation with $\delta=2$.
Then we have the following degree constraints for the two syzygies: 
\[ 
\deg U_1=3,\quad \deg V_1=1,\quad \deg W_1<3,\quad
\deg U_2<3,\quad \deg V_2<1,\quad \deg W_2=3.
\] 
We can take the same syzygy (\ref{eq:syzs32}) for $(U_2,V_2,W_2)$, and derive the same solution
(\ref{eq:y32}) of $P_{VI}(2/5,2/5,2/5,2/5;t_{12})$. We can take $(U_1,V_1,W_1)$ equal to syzygy 
(\ref{eq:syzs32}) times \mbox{$x-s/2-1$}, plus the syzygy $3(s+4)S_3$. Application of expression 
(\ref{eq:llentry}) to this syzygy gives 
the following solution $\widetilde{y}_{32}(t_{12})$ of $P_{VI}(2/5,2/5,2/5,-2/5;t_{12})$:
\begin{eqnarray}
\widetilde{y}_{32}=\frac{(u-1)^2(u+3)^2(3u^2+1)(7u^8-108u^6+314u^4-588u^2+119)}
{7(u+1)^3(u-3)(u^2+3)(u^2+4u-1)(3u^6-37u^4+209u^2+17)}.
\end{eqnarray}


The same covering $z=\varphi_{12}(x)$ can be applied to pullback the Fuchsian equations 
$E(1/3,1/2,0,1/4;t;z)$ and $E(1/3,1/2,0,1/2;t;z)$ to isomondromic matrix equations 
with four singular points.  Let us denote:
\begin{equation}
\lambda(x)=\frac{t^*_{12}x}{x+t^*_{12}-1}. 
\end{equation}
The fractional-linear transformation $\lambda(x)$ fixes the points $x=0$ and $x=1$, 
and moves $x=\infty$ to $x=t^*_{12}$. Theorem \ref{kit:method} can be applied to 
$\widehat{\varphi}_{12}(\lambda(x))$ with $k_0=3$, $k_1=2$, $k_\infty=4$. 
Let us denote $t_{60}=\lambda(t_{12})$ and $y_{61}=\lambda(y_{26})$. Explicitly, we have:
\begin{equation}
t_{60}=\frac{(u-1)(u+3)^3}{(u+1)(u-3)^3},\qquad
y_{61}=\frac{(u+3)^2(u^2-5)}{5(u+1)(u-3)(u^2+3)}.
\end{equation}
In the current application of Theorem \ref{kit:method}, the branches $x=t$ and
$x=y$ are given by, respectively, $x=t_{60}$ and $x=y_{61}$. 
We conclude that $y_{61}(t_{60})$ is a solution of $P_{VI}(1/4,1/4,1/4,-1/4;t_{60})$.
The same solution is given in \cite[pg. 25]{HGBAA}, reparametrized with $u\mapsto (s-3)/(s+1)$.

Currently, the implied $RS$-transformation is
$RS^2_4\left( 1/3 \atop 3+3+3+3 \right|  {1/2 \atop 2+2+2+2+2+2} \left| 1/4 \atop 5+4+1+1+1 \right)$.
To get a solution of $P_{VI}(1/4,1/4,1/4,1/4;t)$, we may use the upper-right
entry of the pullbacked equation.
In order to apply Theorem \ref{th:llentry}, we may substitute $x\mapsto 1/x$ 
in expression (\ref{eq:phi12}) of $\varphi_{12}(x)$. Accordingly, let $\widetilde{F}_{12}(x)$,
$\widetilde{P}_{12}(x)$ denote the polynomials $x^4F_{12}(1/x)$, $x^6P_{12}(1/x)$, respectively. 
A suitable syzygy 
between $\big(\widetilde{F}_{12}(x),\widetilde{P}_{12}(x),x\big)$ is the same as in (\ref{eq:syzs32})
except that the coefficients to $x^2$ and $x$ of the third component have to be interchanged.
The expression as in (\ref{eq:lle0}) is $(s+3)/2$. After application of back substitutions
$x\mapsto 1/x$, (\ref{eq:ts12}) and fractional-linear transformation $\lambda^{-1}$,  (\ref{eq:lambd12})
We get the following solution of $P_{VI}(1/4,1/4,1/4,1/4;t_{62})$:
\begin{equation}
y_{62}=-\frac{(u+3)^2}{3(u+1)(u-3)}.
\end{equation}
The parametrization in \cite[p.588]{Hit} and  \cite[(10)]{Bo3} is related by $u\mapsto-3/(2s-1)$.
Boalch notes that this solution is also equivalent to \cite[(E.29)]{Dubr}.

We may also consider $RS$-transformations
$RS^2_4\left( 1/3 \atop 3+3+3+3 \right|  {1/2 \atop 2+2+2+2+2+2} \left| 1/2 \atop 5+4+1+1+1 \right)$.
We have to compute syzygies between $\big(\widetilde{F}_{12}(x),\widetilde{P}_{12}(x),x^2\big)$.
The ``lower" syzygy $(U_2,V_2,W_2)$ gives a solution of $P_{VI}(1/2,1/2,1/2,5/2;t)$, 
or equivalently, 
$P_{VI}(1/2,1/2,1/2,-1/2;t)$. Incidentally, we get the same function $y_{62}(t_{60})$ as the $z$-root
of the lower-left entry, although the syzygy $(U_2,V_2,W_2)$ is different:
\begin{eqnarray} \textstyle
\left( x^2-2(s+3)x+1,\,-\frac12,\,3(s+4)\!\left(x^3+(2s+11)x^2+(s^2+5s+7)x-\frac12s-1\right) \right).
\end{eqnarray}
Hence, $y_{62}(t_{60})$ is a solution of $P_{VI}(1/2,1/2,1/2,-1/2;t_{60})$ as well.
As for the syzygies $(U_1,V_1,W_1)$ for the upper row of the Schlesinger matrix, we may take 
$\delta=0$ or $\delta=2$, and get the syzygies
\begin{eqnarray} \textstyle
\left(-\frac23x^2-\frac23(s+3)x+\frac13, -\frac16, 
x^4+(3s+16)x^3+(3s^2+25s+58)x^2+\ldots\right),\\ 
\textstyle \left( x^3+(s+7)x^2-(2s+5)x+1, -\frac12(x+1), \frac32(s+4)^2(2x^2+(2s+9)x-1) \right).
\end{eqnarray}
Eventually, we derive these solutions $y_{63}(t_{60})$ and $y_{64}(t_{60})$ of
$P_{VI}(1/2,1/2,1/2,-5/2;t_{60})$ and $P_{VI}(1/2,1/2,1/2,1/2;t_{60})$,
respectively,  
\begin{equation}
y_{63}=-\frac{(u+3)^2(u^2+7)}{7(u+1)(u-3)(u^2+3)},\qquad 
y_{64}= \frac{(u-1)(u+3)^2}{(u-3)(u^2+3)}.
\end{equation}
Algebraic solutions of $P_{VI}(1/2,1/2,1/2,1/2;t)$ are closely investigated in \cite{Hi} and \cite{Hit}.
In particular, the solution $t_{60}/y_{64}$ is presented in \cite[6.4]{Hi} and \cite[pg. 598]{Hit},
reparametrized by $u\mapsto 3(s+1)/(s-1)$. 
The equation $P_{VI}(1/2,1/2,1/2,1/2;t)$ is related to Picard's 
$P_{VI}(0,0,0,1;t)$ via an Okamoto transformation.

\section{Appendix}
\label{sec:appendix}

Recall that the sixth Painlev\'e equation is, canonically,
\begin{eqnarray}
 \label{eq:P6}
\frac{d^2y}{dt^2}&=&\frac 12\left(\frac 1y+\frac 1{y-1}+\frac 1{y-t}\right)
\left(\frac{dy}{dt}\right)^2-\left(\frac 1t+\frac 1{t-1}+\frac 1{y-t}\right)
\frac{dy}{dt}\nonumber\\
&+&\frac{y(y-1)(y-t)}{t^2(t-1)^2}\left(\alpha+\beta\frac t{y^2}+
\gamma\frac{t-1}{(y-1)^2}+\delta\frac{t(t-1)}{(y-t)^2}\right),
\end{eqnarray}
where $\alpha,\,\beta,\,\gamma,\,\delta\in\CC$ are parameters. As well-known
\cite{JM}, its solutions
define isomonodromic deformations (with respect to $t$) of the $2\times2$
matrix Fuchsian equation with 4 singular points ($\lambda=0,1,t$, and
$\infty$):
\begin{equation}
 \label{eq:JM}
\frac{d}{dz}\Psi=\left(\frac{A_0}z+
\frac{A_1}{z-1}+\frac{A_t}{z-t}\right)\Psi,\qquad
\frac{d}{dz}A_k=0\quad\mbox{for } k\in\{0,1,t\}.
\end{equation}
The standard correspondence is due to Jimbo and Miwa \cite{JM}. We choose
the traceless normalization of (\ref{eq:JM}), so we assume that the
eigenvalues of $A_0$, $A_1$, $A_t$ are, respectively, $\pm\theta_0/2$,
$\pm\theta_1/2$, $\pm\theta_t/2$, and that the matrix
$A_\infty:=-A_1-A_2-A_3$ is diagonal with the diagonal entries
$\pm\theta_\infty/2$.  Then the corresponding Painlev\'e equation has the
parameters
\begin{equation}
 \label{eq:para}
\alpha=\frac{(\theta_\infty-1)^2}2,\quad
\beta=-\frac{{\theta}_0^2}2,\quad\gamma=\frac{{\theta}_1^2}2,
\quad\delta=\frac{1-{\theta}_t^2}2.
\end{equation}
We refer to the numbers $\theta_0$, $\theta_1$, $\theta_t$ and
$\theta_\infty$ as {\em local monodromy differences}.

For any numbers $\nu_1,\nu_2,\nu_t,\nu_\infty$, we denote by
$P_{VI}(\nu_0,\nu_1,\nu_t,\nu_\infty;t)$ the Painlev\'e VI equation for the
local monodromy differences $\theta_i=\nu_i$ for $i\in\{0,1,t,\infty\}$, via
(\ref{eq:para}). Note that changing the sign of $\nu_0,\nu_1,\nu_t$ or
$1-\nu_\infty$ does not change the Painlev\'e equation. Fractional-linear
transformations for the Painlev\'e VI equation permute the 4 singular points
and the numbers $\nu_0,\nu_1,\nu_t,1-\nu_\infty$.

Similarly, for any numbers $\nu_1,\nu_2,\nu_t,\nu_\infty$ and a solution $y(t)$ of
$P_{VI}(\nu_0,\nu_1,\nu_t,\nu_\infty;t)$, we denote by $E(\nu_0,\nu_1,\nu_t,\nu_\infty;y(t);z)$
a Fuchsian equation (\ref{eq:JM}) corresponding to $y(t)$ by the Jimbo-Miwa correspondence.
The Fuchsian equation is determined uniquely up to conjugation of $A_0,A_1,A_t$ by
a diagonal matrix (dependent  on $t$ only). In particular, \mbox{$y(t)=t$} can be considered as a
solution of $P_{VI}(e_0,e_1,0,e_\infty;t)$. The equation $E(e_0,e_1,0,e_\infty;t;z)$ is a Fuchsian equation with 3 singular points, actually without the parameter $t$. Its solutions can be expressed
in terms of Gauss hypergeometric series with the local exponent differences $e_0$, $e_1$ and
$e_\infty\pm1$. 
We refer to $E(e_0,e_1,0,e_\infty;t;z)$ as a {\em matrix hypergeometric equation}, and
see it as a matrix form of Euler's ordinary hypergeometric equation.
In particular, the monodromy group of $E(1/3,1/2,0,1/5;t;z)$ or $E(1/3,1/2,0,2/5;t;z)$
is the icosahedral group.

The following matrix form of the hypergeometric equation is considered within
the Jimbo-Miwa correspondence \cite{JM}:
\begin{eqnarray} \label{eq:mhypergeom}
\frac{d}{dz}\Psi=\frac{1}{4e_\infty z\,(1-z)}
\left( \begin{array}{cc} e_0^2-e_1^2+e_\infty^2-2e_\infty^2z&e_\infty^2-(e_0+e_1)^2\\
(e_0-e_1)^2-e_\infty^2 &
2e_\infty^2z-e_0^2+e_1^2-e_\infty^2\end{array}\right) \Psi.
\end{eqnarray}
When considered as a ``constant" isomonodromic system, this equation corresponds to the 
function $y(t)=t $ as a solution of the equation $P_{VI}(e_0,e_1,0,e_\infty;t)$ within
the Jimbo-Miwa correspondence. The function $y(t)=t $ solves $P_{VI}(e_0,e_1,0,e_\infty;t)$
in the following sense: if we multiply both sides of (\ref{eq:P6}) by $y-t$ and simplify each fractional term,
the non-multiples of $y-t$ on the right-hand side form the expression
$\frac12\!\left(\frac{dy}{dt}\right)^2\!-\frac{dy}{dt}+\frac12\frac{y(y-1)}{t(t-1)}$.

Here is a solution of (\ref{eq:mhypergeom}), well defined if $e_0$ is not a positive integer:
\begin{eqnarray} \label{eq:lsolz0}
z^{-\frac12e_0}(1-z)^{-\frac12e_1} \left(\begin{array}{c}
(e_0+e_1-e_\infty)\;
\hpg21{\frac12(-e_0-e_1-e_\infty),\,1+\frac12(-e_0-e_1+e_\infty)}{1-e_0}{\,z\,} \\
(e_0-e_1+e_\infty)\;
\hpg21{\frac12(-e_0-e_1+e_\infty),\,1+\frac12(-e_0-e_1-e_\infty)}{1-e_0}{\,z\,}
\end{array}\right).
\end{eqnarray}
If $e_0\neq 0$, then an independent solution can be obtained by flipping the
sign of $e_0$ and $e_1$ in this expression. (If we would flip the sign of $e_0$ only, 
the Fuchsian equation would be different.) Up to constant multiples, local solutions
at singular points have the following asymptotic first terms:
\begin{eqnarray} \label{eq:asymtz0}
\mbox{at } z=0: && {e_0+e_1-e_\infty\choose
e_0-e_1+e_\infty}\;z^{-\frac12e_0} \qquad\mbox{or}\qquad
{e_0+e_1+e_\infty\choose e_0-e_1-e_\infty}\;z^{\frac12e_0}\,;\\
\mbox{at } z=1: && {e_0+e_1-e_\infty\choose
e_0-e_1-e_\infty}(1-z)^{-\frac12e_1} \quad\mbox{or}\quad
{e_0+e_1+e_\infty\choose e_0-e_1+e_\infty} (1-z)^{\frac12e_1}.
\end{eqnarray}
Hypergeometric solutions at $z=1$ can be obtained from (\ref{eq:lsolz0}) by the substitutions
\mbox{$z\mapsto 1-z$}, $e_0\leftrightarrow e_1$ and applying the matrix $1\;\ 0\choose\ 0\ -1$ 
to the solution vector. 
Due to the normalization, at $z=\infty$ we have a basis of solutions
\begin{equation} \label{eq:solinf}
{1\choose0}\,z^{\frac12e_\infty}+O\left(z^{\frac12e_\infty-1}\right), \qquad
{0\choose1}\,z^{-\frac12e_\infty}+O\left(z^{-\frac12e_\infty-1}\right).
\end{equation}
Explicitly, a hypergeometric basis at $z=\infty$ is 
\begin{eqnarray} \label{eq:lsolzi1}
z^{\frac12(e_1+e_\infty)}(1-z)^{-\frac12e_1} \!\left(\begin{array}{c}
4e_\infty(e_\infty-1)\;
\hpg21{\frac12(-e_0-e_1-e_\infty),\,\frac12(e_0-e_1-e_\infty)}{-e_\infty}{\frac1z\,} \\
\frac{e_\infty^2-(e_0-e_1)^2}z\,
\hpg21{1+\frac12(-e_0-e_1-e_\infty),\,1+\frac12(e_0-e_1-e_\infty)}{2-e_\infty}{\frac1z}\!
\end{array}\right),\\ \label{eq:lsolzi2}
z^{\frac12(e_1-e_\infty)}(1-z)^{-\frac12e_1} \!\left(\begin{array}{c}
\frac{e_\infty^2-(e_0+e_1)^2}z\,
\hpg21{1+\frac12(-e_0-e_1+e_\infty),\,1+\frac12(e_0-e_1+e_\infty)}{2+e_\infty}{\frac1z}\! \\
4e_\infty(e_\infty+1)\;
\hpg21{\frac12(-e_0-e_1+e_\infty),\,\frac12(e_0-e_1+e_\infty)}{e_\infty}{\frac1z\,}.
\end{array}\right).
\end{eqnarray}

\small

\bibliographystyle{alpha}

\end{document}